\documentclass[lettersize,journal]{IEEEtran}
\usepackage{amsmath,amsfonts}
\usepackage{algorithmic}
\usepackage{algorithm}
\usepackage{array}
\usepackage[caption=false,font=normalsize,labelfont=sf,textfont=sf]{subfig}
\usepackage{textcomp}
\usepackage{stfloats}
\usepackage{url}
\usepackage{verbatim}
\usepackage{graphicx}
\usepackage{cite}
\usepackage{amssymb}
\newtheorem{proposition}{Proposition}
\newtheorem{lemma}{Lemma}
\newtheorem{remark}{Remark}
\begin{document}

\title{Gaussianization-Based Parameter Estimation for Gamma–Gamma and Lognormal–Rician Turbulence Channels}

\author{{Maoke Miao, Bo Liu,
        Xinyu Zhang,  and Xiao-Yu Chen}
\thanks{This work was supported in part by the National Natural Science Foundation of China (Grant 61871347, Grant 12405026); in part by the Natural Science Foundation of Zhejiang Province (Grant LQN25F010019)\emph{(Corresponding authors: Maoke Miao)}.\\
\indent Maoke~Miao is with the Foundation Science Education Center, Hangzhou City University, 310015, China. Bo Liu and Xiao-Yu Chen are with the School of Information and Electrical Engineering, Hangzhou City University, Hangzhou 310015, China (e-mail: maokemiao@zju.edu.cn; liubo@hzcu.edu.cn; chenxiaoyu@hzcu.edu.cn;). \\
\indent Xinyu Zhang is with the
Department of Communications and Networking, Aalto University, Espoo 02150, Finland (xinyu.1.zhang@aalto.fi).}}

%

\maketitle

\begin{abstract}
Accurate parameter estimation for atmospheric turbulence channels is challenging because the probability density functions of the Gamma–Gamma (GG) and Lognormal–Rician (LR) models involve special functions and numerical integrations. This paper proposes two Gaussianization parameter estimators for GG and LR turbulence channels, i.e., the quantile-transformation (QT) estimator and the Box–Cox estimator. The QT estimator employs bidirectional cross-transformation together with higher-order statistical matching, whereas the Box–Cox estimator constructs an approximate likelihood by incorporating the transformation Jacobian. Asymptotic analysis identifies skewness as the leading-order deviation from Gaussianity under weak-to-moderate turbulence and yields asymptotic expressions for the Box–Cox power parameters. In addition, a physics-informed regularizer based on the extended Rytov-theory mapping from the Rytov variance to the GG shape parameters is introduced to improve parameter identifiability. Simulation results demonstrate that both estimators provide robust performance under noiseless and noisy conditions, while the physics-informed regularization term can improve the parameter estimation performance by at least three orders of magnitude compared with the iterative moment-based estimator and the method-of-moments/convex-optimization estimator in specific turbulence scenarios.
\end{abstract}


\section{Introduction}
\IEEEPARstart{S}{ixth}-generation (6G) communication networks are envisioned
to provide ubiquitous connectivity by integrating terrestrial,
airborne, and satellite communication networks.
To meet  the  requirements for high-capacity fronthaul/backhaul connectivity,  secure  data transmission,
optical   and quantum communications have
been identified as promising enabling technologies for future
6G networks \cite{chowdhury2020sixg,jeon2023fso6g,
zeydan2025quantum6g}. Free-space optical (FSO) links offer license-free spectrum, strong spatial directivity, and rapid deployment for terrestrial backhaul and non-terrestrial connectivity, while also enabling long-distance quantum communication beyond the reach of direct fiber transmission, as demonstrated by satellite-to-ground QKD \cite{jeon2023fso6g,liao2017satelliteqkd}.
Beyond information transmission, FSO links
have enabled high-precision time and frequency transfer, with
femtosecond-level timing deviation and fractional frequency
instability below \(10^{-18}\) demonstrated over a
kilometer-scale atmospheric link
\cite{giorgetta2013opticaltwoway}.

Despite these advantages, optical propagation through the
atmosphere is inevitably affected by random fluctuations of
the refractive index caused by variations in temperature and
pressure. The resulting atmospheric turbulence gives rise to
turbulence-induced effects such as irradiance scintillation and beam wandering. These effects induce random fluctuations in the received
irradiance, which equivalently characterizes the instantaneous
channel transmittance under the normalized channel model. Accurate statistical characterization of these fluctuations is thus essential for evaluating link performance.

A variety of statistical channel models have been developed to characterize turbulence-induced irradiance fluctuations,
which can be  classified according to the number of
shaping parameters. Single-parameter models, such as the
Lognormal and negative-exponential distributions, offer low
computational complexity but are generally applicable only
to weak turbulence and the scintillation-saturation regime,
respectively \cite{phillips1981measured}.
The
M\'alaga and modified M\'alaga models, commonly characterized by four physical parameters, provide greater modeling flexibility and encompass several classical turbulence models as special cases, but their relatively large parameter sets substantially increase the estimation
complexity \cite{miao2024modifiedmalaga}.
Two-parameter models therefore provide an attractive
compromise between modeling accuracy and estimation
tractability. Among them, the Gamma--Gamma (GG) and
Lognormal--Rician (LR) distributions are particularly prominent \cite{alhabash2001gamma,churnside1989lognormalrician}.

For the GG distribution, direct
maximum-likelihood estimation is challenging because its PDF
contains a modified Bessel function of the second kind, and
the associated optimization may involve derivatives with
respect to both its argument and order. A maximum-likelihood
estimator based on numerical optimization was studied in
\cite{kazeminia2013ggml}, while an expectation-maximization
(EM) algorithm was subsequently developed in
\cite{chen2021ggem}. To avoid direct likelihood evaluation, Wang and Cheng developed a method-of-moments/convex-optimization (MoM/CVX) estimator \cite{wang2010ggestimation}. Their
modified scheme further used the physical relationship between
the two shape parameters to refine one parameter after the
moment-based estimates had been obtained. Generalized
fractional-moment estimators were developed in
\cite{lim2018generalizedgg}, and their limiting optimal forms
were derived in \cite{kim2022optimalgg}. More recently, Kim and Yoon proposed
an iterative weighted-moment estimator (IWME) based on the first five
moments of noisy observations
\cite{kim2025ggnoise}. Although these methods have
considerably advanced GG parameter estimation,
the weak-turbulence regime remains difficult because large
shape parameters may generate very similar distribution
profiles, resulting in weak parameter identifiability and
relatively flat objective-function regions.

Parameter estimation for the LR distribution
is even more challenging because its PDF is expressed through an
integral containing a modified Bessel function of the first
kind. Song and Cheng introduced a generalized method of
moments (GMM) to jointly estimate the coherence parameter and
the Lognormal variance \cite{song2012lrgmm}. This method
avoids explicit likelihood integration and can be applied to
both noiseless and noisy observations, but satisfactory
performance may require up to \(10^6\) acquired samples, leading to an excessive data-acquisition
latency that is unsuitable for real-time communication systems.
An EM maximum-likelihood estimator was developed in
\cite{yang2015lrem}, which reduces the required observation
size at the cost of repeated numerical integration. A
saddlepoint approximation (SAP) was employed in
\cite{miao2020lrsap} to replace the original integral-form
likelihood with a tractable approximation, although
saddlepoint calculations and Bessel-function evaluations
remain necessary. More recently, the authors developed a data-generation estimator using $k$-nearest-neighbor density approximation \cite{Miao25}, thereby avoiding direct evaluation of the integral-form PDF by constructing a nonparametric likelihood approximation from generated samples.

Building on the  data-driven nature of the $k$-nearest-neighbor
estimator, we ask whether the observations can be mapped into a more tractable domain to facilitate parameter estimation.
The Gaussian domain is a natural choice because Gaussian
log-likelihoods and the Kullback-Leibler divergence between
Gaussian distributions can be readily evaluated in closed form.
Moreover, Gaussian reference distributions are widely used
in modern generative modeling, where complex distributions
are constructed from or progressively transformed toward
Gaussian variables, as exemplified by variational autoencoders
and denoising diffusion probabilistic models
\cite{kingma2014vae,ho2020ddpm}. Motivated by these
properties, we develop two Gaussianization parameter estimators, namely, the QT and Box–Cox estimators, for GG and
LR turbulence channels in the presence of
additive Gaussian noise.  The main contributions are
summarized as follows:

\begin{itemize}
\item  We first investigate the applicability of the QT and Box-Cox transformations to GG and LR turbulence channels. Higher-order analysis identifies skewness as the leading-order deviation from Gaussianity under weak-to-moderate turbulence, from which asymptotic expressions for the Box-Cox power parameters are derived. The Gaussianization accuracy of both transformations is evaluated using Kolmogorov--Smirnov (KS) tests for noiseless and noisy observations.

\item We develop two corresponding parameter estimators. The
QT estimator employs bidirectional cross-transformation and
higher-order statistical matching, whereas the Box–Cox
estimator constructs an approximate likelihood using the
transformation Jacobian.  To the best of our knowledge, this is the first study of Gaussianization-based parameter estimators for both considered turbulence channels, encompassing both the nonparametric QT and parametric Box–Cox approaches.

\item  We reveal that different GG shaping parameters can yield similar distributions, particularly under weak turbulence, thereby weakening parameter identifiability. To address this issue, we introduce a physics-informed regularizer based on the Rytov-variance-induced relationship between the two parameters. To the best of our knowledge, this is the first use of this relationship as a physics-informed regularizer in the objective function for GG parameter estimation. Under the considered turbulence conditions, the introduced regularization term can improve the estimation performance by at least three orders of magnitude compared with the IWME and the MoM/CVX in specific turbulence
scenarios.

\end{itemize}

The remainder of this paper is organized as follows. Section II presents the two turbulence-channel models and analyzes their statistical properties, physical parameter relationships, and Gaussianization performance under the QT and Box–Cox transformations. Section III develops the proposed estimators, including physics-informed regularization for GG parameter estimation. Section IV evaluates their performance under different sample sizes, turbulence conditions, and noise levels and compares them with benchmark estimators. Finally, Section V concludes the paper.

\section{Preliminaries}

\subsection{FSO Channel Models}
\indent Let \(I\) denote the normalized received irradiance such that
\(\mathbb{E}[I]=1\). To characterize the combined effects of large- and small-scale atmospheric fluctuations, we consider two widely used turbulence models, the GG  and LR distributions. Their statistical definitions and relevant
properties are presented below. \\
\indent 1) \emph{GG distribution}: According to \cite{alhabash2001gamma}, the PDF of the GG random variable $I$ and its $k$-th moment $\mu_k$ are respectively given by
\begin{equation}\label{eq1}
  f_{G G}(I;\alpha,\beta)=\frac{2(\alpha \beta)^{\frac{\alpha+\beta}{2}}}{ \Gamma(\alpha) \Gamma(\beta)}\left(I\right)^{\frac{\alpha+\beta}{2}-1} K_{\alpha-\beta}\left(2 \sqrt{\alpha \beta I}\right),
\end{equation}
\begin{equation}\label{eqAdd1}
  \mu_{GG,k}=\mathbb{E}\left[I^k\right]=\frac{\Gamma(\alpha+k) \Gamma( \beta+k)}{\Gamma(\alpha) \Gamma(\beta)}\left(\frac{1}{\alpha \beta}\right)^k,
\end{equation}
where $\alpha$ and $\beta$ are shape parameters representing the effective number of large-scale and small-scale turbulent cells, respectively, $\Gamma\left(\cdot\right)$ is the Gamma function \cite[eq. (8.310.1)]{gradshteyn2007}, and $K_v\left(\cdot\right)$ is the $v$-th order
modified Bessel function of the second kind \cite[eq. (8.407.1)]{gradshteyn2007}. It can be found from  (\ref{eq1}) that interchanging $\alpha$ and $\beta$ gives the same model due to the symmetry property $K_{\alpha - \beta}\left(\cdot\right) = K_{\beta - \alpha}\left(\cdot\right)$. Therefore, without loss of generality, we assume $\alpha\ge\beta$. \\
\indent According to \cite{Chattamvelli2023}, skewness and excess kurtosis are commonly employed to characterize higher-order deviations of a given PDF from the normal Gaussian distribution, where skewness quantifies distributional asymmetry, whereas excess kurtosis quantifies the relative heaviness of the distribution tails. For the  GG distribution, they are defined as the standardized third-order central moment and the standardized fourth-order central moment minus three, respectively, i.e.,
\begin{equation}\label{eqAdd1_1}
  \gamma_{1, \mathrm{GG}}=\mathbb{E}\left[\left(\frac{I-\mu_{\mathrm{GG}}}{\sigma_{\mathrm{GG}}}\right)^3\right],\gamma_{2, \mathrm{GG}}=\mathbb{E}\left[\left(\frac{I-\mu_{\mathrm{GG}}}{\sigma_{\mathrm{GG}}}\right)^4\right]-3,
\end{equation}
where $\mu_{GG}$ and $\sigma_{GG}$ denote the mean and standard deviation of the GG distribution, respectively. Using  (\ref{eqAdd1}), it can be readily verified that $\gamma_{1,GG} > 0$ for all $\alpha$ and $\beta$, indicating  the GG distribution is positively skewed.  In particular, as established in the following proposition, skewness becomes the dominant higher-order contribution in the relatively large-parameter regime, corresponding to a trend toward
weaker turbulence.
\begin{proposition}
\label{propskewness1}
For sufficiently large $\alpha, \beta$, the skewness
and excess kurtosis of the GG distribution are
asymptotically expressed as
\begin{equation}
\gamma_{1,\mathrm{GG}}
=
\frac{2\left(\tau^{2}+3\tau+1\right)}
{\sqrt{\tau}\left(\tau+1\right)^{3/2}}
\frac{1}{\sqrt{\beta}}
+
\mathcal{O}\!\left(\beta^{-3/2}\right),
\label{eq:GG_skewness_asymptotic}
\end{equation}
and
\begin{equation}
\gamma_{2,\mathrm{GG}}
=
\frac{6\left(\tau^{2}+5\tau+1\right)}
{\tau\left(\tau+1\right)}
\frac{1}{\beta}
+
\mathcal{O}\!\left(\beta^{-2}\right),
\label{eq:GG_kurtosis_asymptotic}
\end{equation}
where $\tau=\alpha/\beta \geq 1$, and $\mathcal{O}(\beta^{-p})$ denotes
the asymptotic order of a remainder term whose magnitude is
bounded by a constant multiple of $\beta^{-p}$ \cite{Hentila2024}.
\end{proposition}
\begin{IEEEproof}
The proof is provided in Appendix~A.
\end{IEEEproof}

\indent  Furthermore, according to \cite{Ahmadi10}, when either shape parameter becomes sufficiently large, the GG distribution can be well approximated by a simpler Gamma distribution. Consequently, the GG PDF becomes progressively less sensitive to further increases in the larger parameter, such that different parameter combinations may yield similar PDF shapes. This behavior is expected to produce a relatively flat Kullback-Leibler divergence surface along the direction of the increasing shaping parameter. It should be emphasized that this intrinsic weak-identifiability property  of the GG model has not been  demonstrated in previous  studies. \\
\begin{figure}[h]
  \centering
  \includegraphics[width=3.5in]{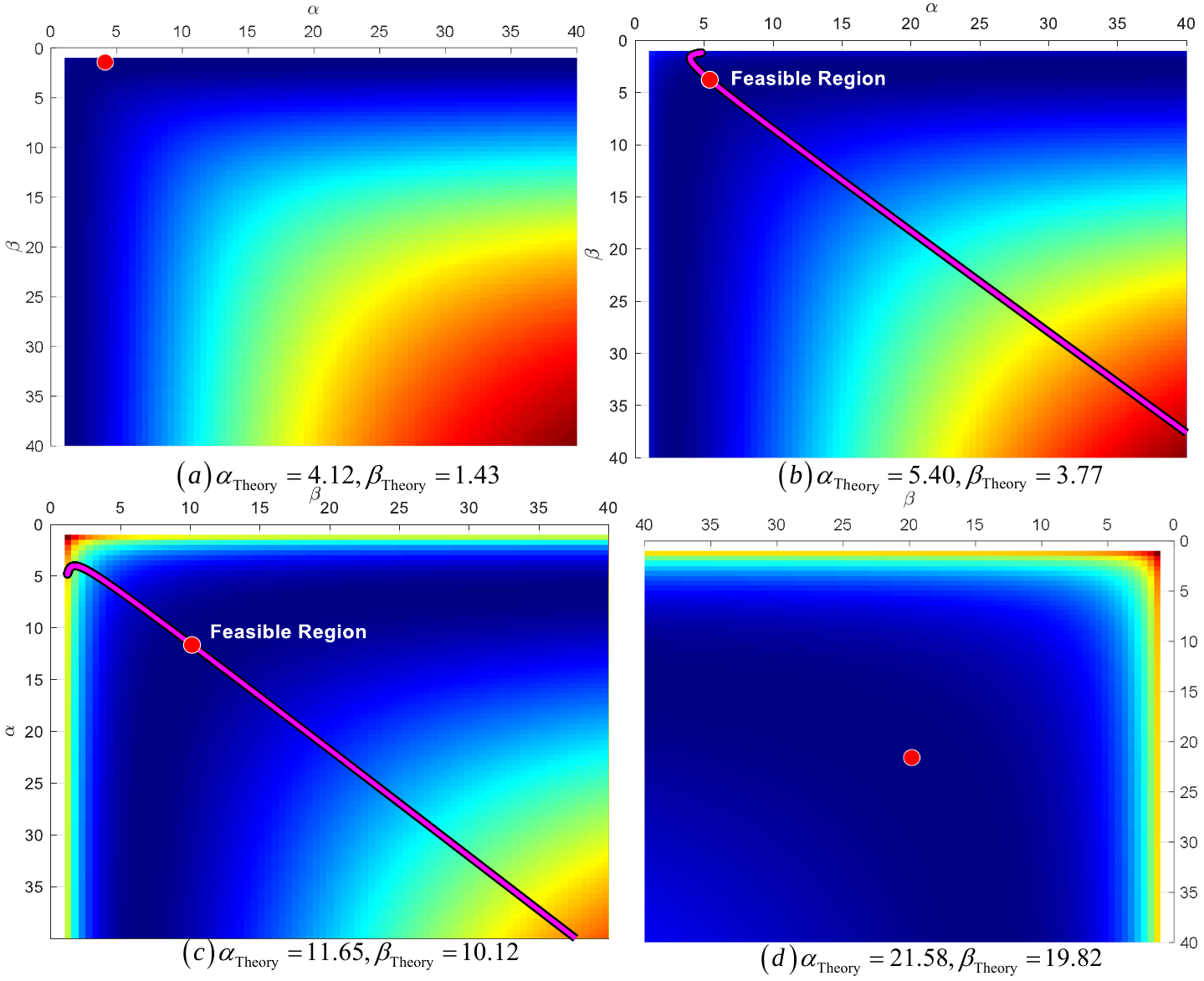}
  \caption{  KL divergence heatmaps for different combinations of shaping parameters $\alpha, \beta$. The corresponding $\sigma_R^2$ values  for (a)-(d) are $3, 0.6, 0.2, 0.1$ respectively, covering turbulence strengths from strong to weak regimes. The red dots mark the theoretical parameter values $(\alpha_{\text{Theory}},\beta_{\text{Theory}})$.}\label{Figure1}
\end{figure}
\indent To illustrate this behavior, Fig.~\ref{Figure1} presents the KL-divergence heatmaps over different combinations of the candidate shaping parameters $(\alpha,\beta)$. Given a theoretical parameter pair $(\alpha_{\mathrm{Theory}},\beta_{\mathrm{Theory}})$, each heatmap value is numerically computed as
\begin{equation}\label{eqAdd2}
D_{\mathrm{KL}}\left(f_{G G}\left(I ; \alpha_{\mathrm{Theory}}, \beta_{\mathrm{Theory}}\right) \| f_{G G}(I ; \alpha, \beta)\right).
\end{equation}
 \indent As shown in Fig.~\ref{Figure1}, the warmer colors (toward red) indicate larger KL divergence values and the red dots in four subplots represent the minimum KL divergence. We can see that  as $\alpha$ or $\beta$ increases, the KL divergence heatmap  continues to exhibit a deep blue similar to the minimum value regions,  thereby confirming our aforementioned inference. In addition, as the theoretical values  $\alpha,$  $\beta$ increase, the expansion of the deep blue regions suggests that a wider range of parameter combinations become potential candidates for the estimation results. In this case, multiple candidate parameter pairs can yield comparable objective values, reducing the identifiability of the true solution. Consequently, the estimator becomes more sensitive to initialization. \\
\indent To mitigate this effect, additional information must be incorporated into the estimation process. According to \cite{Andrews99}, the extended Rytov theory can be regarded as a physically motivated  framework for modeling scintillation in atmospheric turbulence. Within this framework, the scintillation index $\sigma_{SI}^2$  is decomposed into large-scale and small-scale components by modifying the atmospheric spectrum with spatial filters, which is expressed as follows \cite{Andrews2023}
\begin{equation}\label{eq2Add1}
\begin{aligned}
  \sigma_{SI}^2&= \left(1 +\sigma_X^2\right)\left(1 + \sigma_Y^2\right) - 1 = \exp \left(\sigma_{\ln X}^2+\sigma_{\ln Y}^2\right)-1 \\
  &=\exp \left[\frac{0.49 \sigma_R^2}{\left(1+1.11 \sigma_R^{12 / 5}\right)^{7 / 6}}+\frac{0.51 \sigma_R^2}{\left(1+0.69 \sigma_R^{12 / 5}\right)^{5 / 6}}\right] - 1,
\end{aligned}
\end{equation}
for a plane wave with a negligible inner scale, where in (\ref{eq2Add1}),
\begin{equation}\label{eq2Add1_1}
  \sigma_X^2 = \exp \left(\sigma_{\ln X}^2\right)-1,\sigma_Y^2 = \exp \left(\sigma_{\ln Y}^2\right)-1.
\end{equation}
\indent Note that  for the GG distribution, the  scintillation index derived from (\ref{eqAdd1}) is
\begin{equation}\label{eq2Add2}
 \sigma_{SI,GG}^2 = E\left[I^2\right]-1 = \left(1+\frac{1}{\alpha}\right)\left(1+\frac{1}{\beta}\right)-1.
\end{equation}\\
 Thus, by comparing  (\ref{eq2Add1}) and (\ref{eq2Add2}), the parameters $\alpha,\beta$ must satisfy the constraints given by
\begin{equation}\label{eq2}
  \begin{aligned}
& \alpha=g\left(\sigma_R\right)=\left[\exp \left(\frac{0.49 \sigma_R^2}{\left(1+1.11 \sigma_R^{12 / 5}\right)^{7 / 6}}\right)-1\right]^{-1} \\
& \beta=h\left(\sigma_R\right)=\left[\exp \left(\frac{0.51 \sigma_R^2}{\left(1+0.69 \sigma_R^{12 / 5}\right)^{5 / 6}}\right)-1\right]^{-1}.
\end{aligned}
\end{equation}

From (\ref{eq2}), it can be readily verified  $\beta$ decreases monotonically with  $\sigma_R^2$ and this  relationship is subsequently exploited to construct the physics-informed regularization term in \eqref{eq:GG_physical_loss}. Figs.~1(b) and 1(c) present the corresponding physics-informed feasible regions for two representative theoretical parameter settings. The magenta curves represent the feasible parameter sets defined by (\ref{eq2}). In both cases, the feasible curves traverse regions where the KL divergence changes rapidly, thereby providing informative guidance for the optimization and facilitating estimator convergence. \\
\indent 2) \emph{LR distribution}: The LR  model is the product of Lognormal and
Rician distributions, whose PDF
is given by \cite{Andrews2023,Miao25}
\begin{equation}\label{eq3}
  \begin{aligned}
f_{LN}\left(I ; r, \sigma_z^2\right)= & \frac{(1+r) e^{-r}}{\sqrt{2 \pi} \sigma_z} \int_0^{\infty} \frac{d z}{z^2} I_0\left(2\left[\frac{(1+r) r}{z} I\right]^{1 / 2}\right) \\
& \times \exp \left(-\frac{1+r}{z} I-\frac{1}{2 \sigma_z^2}\left(\ln z+\frac{1}{2} \sigma_z^2\right)^2\right),
\end{aligned}
\end{equation}
where in (\ref{eq3}), $r$ is the coherence parameter, $z$ is the Lognormal random variable,  $I_0$ is the zero-order
modified Bessel function of the first kind \cite[Eq.~(8.406.1)]{gradshteyn2007}, and $\sigma_z^2$ is the variance of the logarithm of the irradiance modulation
factor $z$. The  $k$-th moment $\mu_{LN,k}$ of the LR channel  is obtained as
\begin{equation}\label{eq4}
  \mu_{LN,k}=\frac{(k!)^2}{(1+r)^k} \exp \left[\frac{k(k-1) \sigma_z{ }^2}{2}\right] \sum_{m=0}^k \frac{r^m}{(k-m)!(m!)^2} .
\end{equation}
\indent Similarly,  the skewness and excess
kurtosis of the LR distribution are respectively given by
\begin{equation}\label{eq4Add1}
  \gamma_{1, \mathrm{LR}}=\mathbb{E}\left[\left(\frac{I-\mu_{\mathrm{LR}}}{\sigma_{\mathrm{LR}}}\right)^3\right],\gamma_{2, \mathrm{LR}}=\mathbb{E}\left[\left(\frac{I-\mu_{\mathrm{LR}}}{\sigma_{\mathrm{LR}}}\right)^4\right]-3.
\end{equation}
According to Appendix B, the LR
distribution is positively skewed  for all $r>0$ and
$\sigma_z^2\geq 0$. The following proposition shows
that skewness becomes the dominant higher-order contribution
as $r$ increases and $\sigma_z^2$ decreases.
\begin{proposition}
\label{propskewness2}
For sufficiently large $r$ and sufficiently small
$\sigma_z^2$, the skewness and excess kurtosis of the
LR distribution are asymptotically expressed as
\begin{equation}
\gamma_{1,\mathrm{LR}}
=
3\sqrt{V_0}\left(1-2\rho^2\right)
+
\mathcal{O}\!\left(V_0^{3/2}\right),
\label{eq:LR_skewness_asymptotic}
\end{equation}
and
\begin{equation}
\gamma_{2,\mathrm{LR}}
=
8V_0\left(2-9\rho^2+5\rho^3\right)
+
\mathcal{O}\!\left(V_0^2\right),
\label{eq:LR_kurtosis_asymptotic}
\end{equation}
where
\begin{equation}
V_0=\sigma_z^2+\frac{2}{1+r},
0< \rho=\frac{1}{(1+r)V_0} < 1/2.
\label{eq:LR_V0_rho}
\end{equation}
\end{proposition}
\begin{IEEEproof}
The proof is provided in Appendix~C.
\end{IEEEproof}
\indent Moreover, using  (\ref{eq4}), the corresponding scintillation index is obtained as
\begin{equation}\label{eq5}
\begin{aligned}
  \sigma_{SI,LN}^2 &=\exp \left(\sigma_z^2\right) \frac{r^2+4 r+2}{(1+r)^2} - 1 \\
  & = \left(1 + \exp \left(\sigma_z^2\right)-1\right) \left(1+\frac{2r+1}{(1+r)^2}\right) - 1
\end{aligned}
\end{equation}
 Thus, by comparing  (\ref{eq2Add1}) and (\ref{eq5}), the relationships between $r,\sigma_z^2$ and $\sigma_R^2$ are obtained as
 \begin{equation}\label{eq6}
  \begin{aligned}
  \sigma_z^2&=\frac{0.49 \sigma_R^2}{\left(1+1.11 \sigma_R^{12 / 5}\right)^{7 / 6}},
  r=\frac{\left(1-\sigma_Y^2\right)+\sqrt{1-\sigma_Y^2}}{\sigma_Y^2}\\
  \sigma_Y^2&=\exp \left(\frac{0.51 \sigma_R^2}{\left(1+0.69 \sigma_R^{12 / 5}\right)^{5 / 6}}\right)-1
  \end{aligned}
 \end{equation}
 \indent To the best of our knowledge, the relationships in (\ref{eq6}) have not been  derived previously. It can be easily shown that $r$ decreases monotonically with $\sigma_R^2$, whereas $\sigma_z^2$ first increases, reaches a maximum, and then decreases as $\sigma_R^2$ increases. Note that these qualitative trends are consistent with those illustrated in \cite{Churnside1987LognormalRician}.\\
 \indent However, unlike the GG distribution, the LR distribution generally cannot be approximated by using  a single Lognormal or Rician distribution. This non-degeneracy alleviates the statistical ambiguity between its distribution profile and underlying parameters. Fig.~\ref{Figure3} shows the KL divergence heatmaps
with physics-informed constrained feasible regions under two turbulence conditions $\sigma_R^2 = 0.6,3$, and  each heatmap value  is numerically computed as
\begin{equation}
D_{\mathrm{KL}}\!\left(
f_{LR}(I;r_{\mathrm{Theory}},\sigma_{z,\mathrm{Theory}}^2)
\,\middle\|\,
f_{LR}(I;r,\sigma_z^2)
\right).
\end{equation}
 It can be observed from Fig.~\ref{Figure3} that the influence of the physical constraint depends strongly on the turbulence environment. When $\sigma_R^2 = 0.6$ the physically feasible curve mainly lies in a relatively flat low-divergence region and does not clearly pass through the region where the KL divergence changes rapidly. In this case, the constraint provides limited guidance for improving convergence. In contrast, when $\sigma_R^2 = 3$, the feasible curve intersects a region with stronger KL-divergence variation, indicating that the physical constraint offers more useful guidance.
\begin{figure}[h]
  \centering
  \includegraphics[width=3.5in]{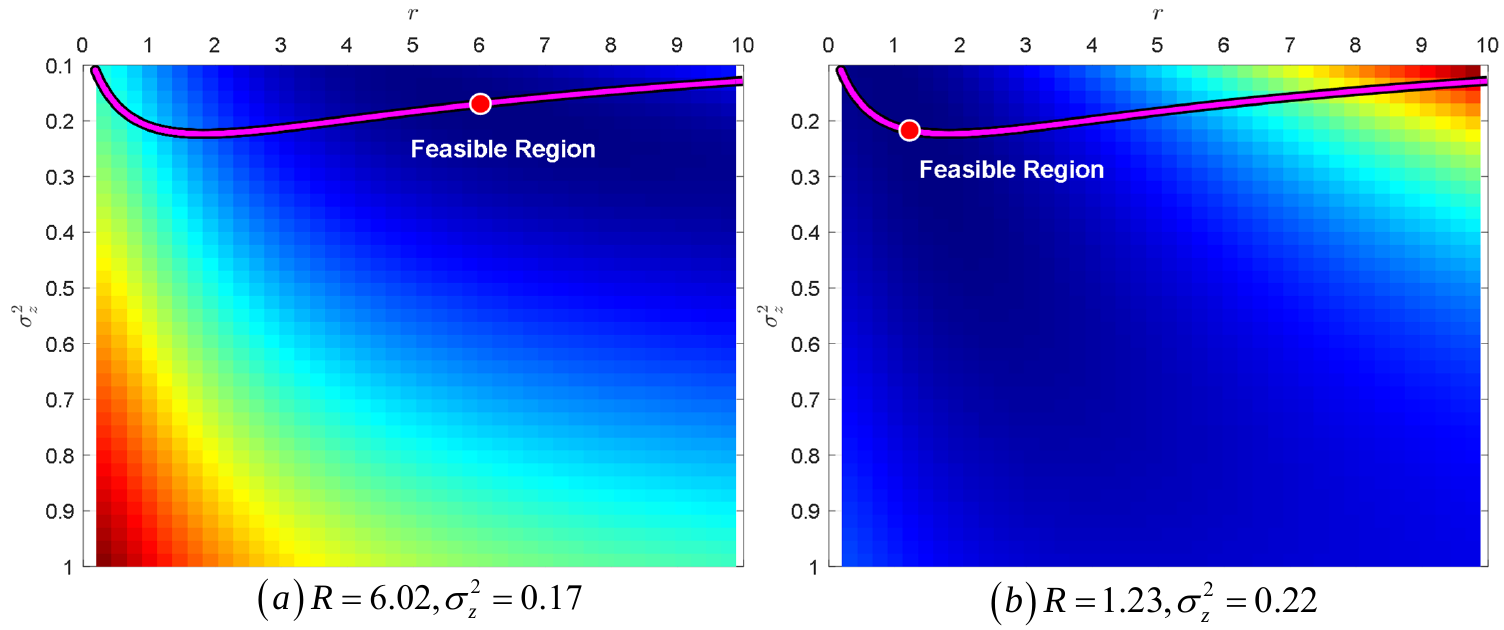}
  \caption{KL Divergence heatmaps with physics-informed constrained feasible
regions. The red dots mark the theoretical parameter values $(r_{\text{Theory}},\sigma^2_{z,\text{Theory}})$.}\label{Figure3}
\end{figure}
\subsection{Gaussianization and KS Goodness-of-Fit Test Results}
\indent According to (\ref{eq1}) and (\ref{eq3}), the probability density functions of the original GG and LR distributions involve special functions and integral operations, making direct likelihood evaluation computationally demanding. To avoid this difficulty, Gaussianizing transformations, including the QT and the  Box–Cox transformation, are applied to convert the irradiance samples into approximate Gaussian samples. In what follows, $Y$ denotes the observed  variable, and we  consider the general noisy observation model, where  $Y$ is given by
\begin{equation}\label{eq7Add1}
  Y = I + N_g,
\end{equation}
 where $I$ represents the turbulence-induced irradiance and $N_g$ represents the additive Gaussian noise with zero mean and  variance $\sigma_g^2$, and they are assumed to be independent of each other. In the following, the signal-to-noise ratio (SNR) is defined as
 \begin{equation}
   \mathrm{SNR}=\frac{(\mathbb{E}[I])^2}{\sigma_g^2} = \frac{1}{\sigma_g^2},
 \end{equation}
 if not otherwise specified. It is easy to show that the third central
moment of $Y=I+N_g$ satisfies
\begin{equation}
\mathbb{E}\left[(Y-\mathbb{E}[Y])^3\right]=\mathbb{E}\left[(I-\mathbb{E}[I])^3\right].
\end{equation}
Moreover, $
\operatorname{Var}(Y)
=
\operatorname{Var}(I)+\sigma_g^2.
$ Therefore,  $Y$ is positively skewed under noisy conditions according to
\begin{equation}
\gamma_Y
=
\gamma_I
\left(
\frac{\operatorname{Var}(I)}
{\operatorname{Var}(I)+\sigma_g^2}
\right)^{3/2} > 0.
\end{equation}

\indent 1) \emph{QT Method:} The QT method is based on the probability integral transform that maps non-Gaussian irradiance samples into approximate Gaussian samples \cite{casella2002statistical}. Notably, this method does not require the input data to be strictly positive, so it remains applicable to noisy observations. For an observed sample $Y$  with cumulative distribution function (CDF) $F_Y(y)$, the transformed variable is defined as
\begin{equation}\label{eq8}
Z_{QT}=\Phi^{-1}(F_Y(Y)),
\end{equation}
where $\Phi^{-1}(\cdot)$ denotes the inverse cumulative distribution function of the standard Gaussian distribution, i.e., $\mathcal{N}(0,1)$, and can be accurately computed using standard numerical routines, such as MATLAB's \emph{norminv} function. In practice, $F_Y(y)$ is replaced by the empirical distribution function estimated from the samples. Hence, given $N$ data samples $\{Y_i\}_{i=1}^{N}$ sorted in ascending order, the empirical distribution  of the $i$-th sorted sample is
\begin{equation}\label{eq9}
p_i=\frac{i-0.5}{N}, \quad i=1,2,\ldots,N,
\end{equation}
according to the plotting rule discussed in \cite{blom1958statistical}. The transformed samples $z_{QT} = \{z_{QT,i}\}_{i=1}^N$ are then obtained as
\begin{equation}\label{eq10}
z_{QT,i}=\Phi^{-1}(p_i), \quad i=1,2,\ldots,N.
\end{equation}
\indent 2) \emph{Two-parameter Box–Cox:} According to \cite{box1964analysis},
the Box–Cox transformation is a parametric Gaussianizing method that applies a power parameter to reduce skewness and stabilize the variance of the data, thereby rendering the transformed samples approximately Gaussian. For the noisy data samples, we adopt the two-parameter Box–Cox
transformation. The transformed variable $Z_{\mathrm{Box-Cox}}$ is defined as
\begin{equation}
Z_{\mathrm{Box-Cox}}
=
\begin{cases}
\dfrac{(Y+\alpha)^\lambda-1}{\lambda},
& \lambda\neq 0,\\[8pt]
\ln(Y+\alpha),
& \lambda=0,
\end{cases}
\end{equation}
Accordingly, for the observations $\{y_i\}_{i=1}^{N}$, the
corresponding transformed samples
${z}_{\mathrm{Box-Cox}}
=
\{z_{\mathrm{Box-Cox},i}\}_{i=1}^{N}$
are given by
\begin{equation}
z_{\mathrm{Box-Cox},i}
=
\begin{cases}
\dfrac{(y_i+\alpha)^\lambda-1}{\lambda},
& \lambda\neq 0,\\[8pt]
\ln(y_i+\alpha),
& \lambda=0,
\end{cases}
\qquad i=1,2,\ldots,N.
\end{equation}
 The shift $\alpha$ guarantees a positive argument, $y_i+\alpha>0$ for
all $i$, while $\lambda$ controls the strength of the power
transformation. Specifically, when $\alpha = 0$, the two-parameter Box–Cox reduces to  the one-parameter Box–Cox
transformation. We note  this case corresponds to the noiseless situation, i.e., $Y = I$. In addition, when $\lambda=1$,
the transformed variable becomes $Y+\alpha-1$, which differs
from $Y$ only by a translation and therefore has the same
skewness and excess kurtosis.  \\
\indent Based on Propositions~\ref{propskewness1} and
\ref{propskewness2}, skewness is the dominant higher-order
contribution under the considered parameter regimes. Thus, if
the Box–Cox transformation effectively improves the
Gaussianity of the channel distribution, its primary effect
should be to suppress the skewness component. Accordingly, the
Box–Cox parameter $\lambda$ can be theoretically determined by
eliminating the leading-order skewness, as stated in the
following proposition.
\begin{proposition}
\label{prop:theoretical_lambda}
Under the parameter regimes considered in
Propositions~\ref{propskewness1} and \ref{propskewness2}, the
skewnesses of the logarithmically transformed GG
and LR distributions satisfy
\begin{equation}
\gamma_{1,\log,\mathrm{GG}}<0,
\quad
\gamma_{1,\log,\mathrm{LR}}<0.
\end{equation}
Moreover, the unique values of $\lambda$ that
eliminate the  skewness are
\begin{equation}
\lambda_{\mathrm{GG}}^{\star}
=
\frac{\tau^2+1}{3(\tau+1)^2},\quad \frac{1}{6} < \lambda_{\mathrm{GG}}^{\star}<\frac{1}{3},
\label{eq:GG_theoretical_lambda}
\end{equation}
for the GG distribution, and
\begin{equation}
\lambda_{\mathrm{LR}}^{\star}
=
\frac{2}
{\left[2+(1+r)\sigma_z^2\right]^2},\quad 0 < \lambda_{\mathrm{LR}}^{\star}<\frac{1}{2},
\label{eq:LR_theoretical_lambda}
\end{equation}
for the LR distribution.
\end{proposition}
\begin{IEEEproof}
The proof is provided in Appendix~D.
\end{IEEEproof}
\begin{remark}
We note that the limiting values in
\eqref{eq:GG_theoretical_lambda} and
\eqref{eq:LR_theoretical_lambda} can be interpreted in terms
of classical Gaussianizing transformations and the dominant
component distributions. The Wilson--Hilferty result shows
that the cube-root transformation provides a Gaussian
approximation for a Gamma variable \cite{Noguchi2026}. As
$r\to\infty$, the GG distribution approaches a
Gamma distribution, and hence
$\lambda_{\mathrm{GG}}^\star\to1/3$. When
$\alpha=\beta$, the two Gamma components contribute
symmetrically, yielding
$\lambda_{\mathrm{GG}}^\star=1/6$.

For the LR distribution, the limiting value is
determined by the dominant component. Let
$\xi=(1+r)\sigma_z^2$, as $\xi\to0$, the Rician component
dominates and $\lambda_{\mathrm{LR}}^\star\to1/2$,
corresponding to the square-root transformation. As
$\xi\to\infty$, the Lognormal component dominates and
$\lambda_{\mathrm{LR}}^\star\to0$, corresponding to the
logarithmic transformation, which maps a Lognormal variable
exactly to a Gaussian variable.
\end{remark}

\begin{figure}[h]
  \centering
  \includegraphics[width=3.5in]{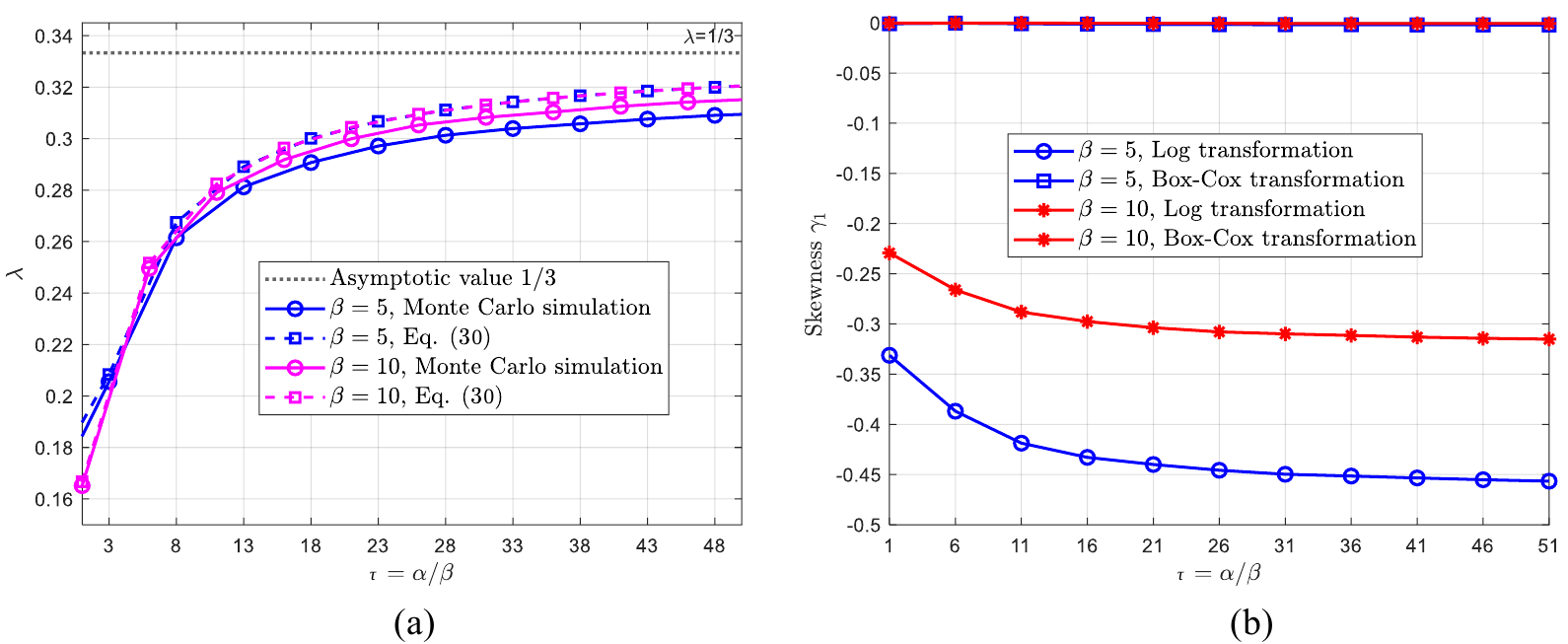}
  \caption{Theoretical and Monte Carlo values of $\lambda$ and
the corresponding skewness for the GG distribution:
(a) theoretical and Monte Carlo values of $\lambda$;
(b) skewness after the logarithmic and Box–Cox transformations.}\label{Figure6}
\end{figure}

\begin{figure}[h]
  \centering
  \includegraphics[width=3.5in]{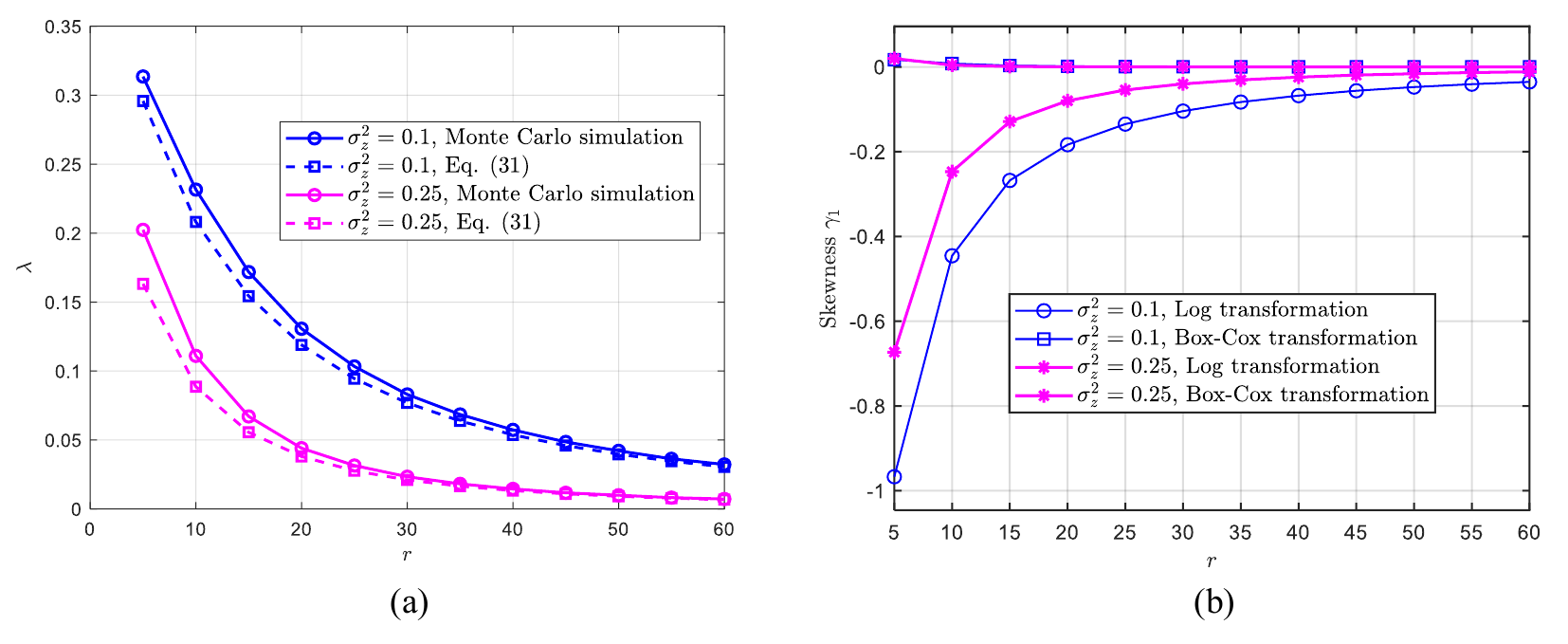}
  \caption{Theoretical and Monte Carlo values of $\lambda$ and
the corresponding skewness for the LR distribution:
(a) theoretical and Monte Carlo values of $\lambda$;
(b) skewness after the logarithmic and Box–Cox transformations.}
\label{fig:LR_lambda_validation}\label{Figure7}
\end{figure}
In Figs.~\ref{Figure6} and \ref{Figure7}, we present numerical
results to validate Proposition~\ref{prop:theoretical_lambda}.
For each parameter setting, the Monte Carlo estimates of
$\lambda$ and skewness are computed using $10^6$ samples and
averaged over 100 independent trials.

Figs.~\ref{Figure6}(a) and \ref{Figure7}(a) demonstrate close
agreement between the theoretical values of $\lambda$ in
\eqref{eq:GG_theoretical_lambda} and
\eqref{eq:LR_theoretical_lambda} and their Monte Carlo
counterparts. In particular, the GG values converge
to the asymptotic limit $1/3$ as $\tau$ increases,
whereas the theoretical and Monte Carlo values for the
LR distribution become increasingly close as
$r$ increases. In addition,  Figs.~\ref{Figure6}(b) and \ref{Figure7}(b)
show that the skewness of both distributions is reduced to
nearly zero through the Box–Cox transformations. It is also worth noting that, although
Proposition~\ref{prop:theoretical_lambda} is derived for
sufficiently large GG shape parameters and for large
$r$ and small $\sigma_z^2$ in the LR model, the
numerical results still show reasonable accuracy for moderate
parameter settings, such as $\alpha=\beta=5$ and
$r=5, \sigma_z^2=0.1$.\\
\indent 3) \emph{KS test results}: Before applying the QT and the Box–Cox transformation to develop  parameter estimation methods, we first use a two-sample KS test to examine whether, for each of the two considered models, the observations transformed using the QT and Box–Cox transformations can be adequately approximated by Gaussian distributions. \\
\indent According to \cite{Monahan2011}, the test statistic is defined as
\begin{equation}
D_{n,m}
=
\sup_x
\left|
F_{1,N}(x)-F_{2,M}(x)
\right|.
\end{equation}
where $N$ and $M$ denote the numbers of transformed observations
and Gaussian reference samples, respectively, and
$\sup$ is the supremum function.
$F_{1,N}(x)$ is the empirical CDF
of the transformed observations, $\{z_i\}_{i=1}^{N}$, whereas $F_{2,M}(x)$ is that
of the Gaussian reference samples. Note that the reference samples are
independently generated from a Gaussian distribution
$\mathcal{N}(\widehat{\mu}_z,\widehat{\sigma}_z^2)$, where
$\widehat{\mu}_z$ and $\widehat{\sigma}_z^2$ are obtained as
\begin{equation}
  \widehat{\mu}_z=\frac{1}{N} \sum_{i=1}^N z_i, \quad \widehat{\sigma}_z^2=\frac{1}{N-1} \sum_{i=1}^N\left(z_i-\widehat{\mu}_z\right)^2,
\end{equation}
\indent The null hypothesis  $H_0$ asserts that the transformed observations and the
Gaussian reference samples are drawn from the same underlying distribution.
For sufficiently large sample sizes, $H_0$ is rejected at the significance
level $\alpha=0.05$ if
\begin{equation}\label{eq26}
\sqrt{\frac{NM}{N+M}}D_{N,M}>1.358,
\end{equation}
\begin{figure}[h]
  \centering
  \includegraphics[width=3.5in]{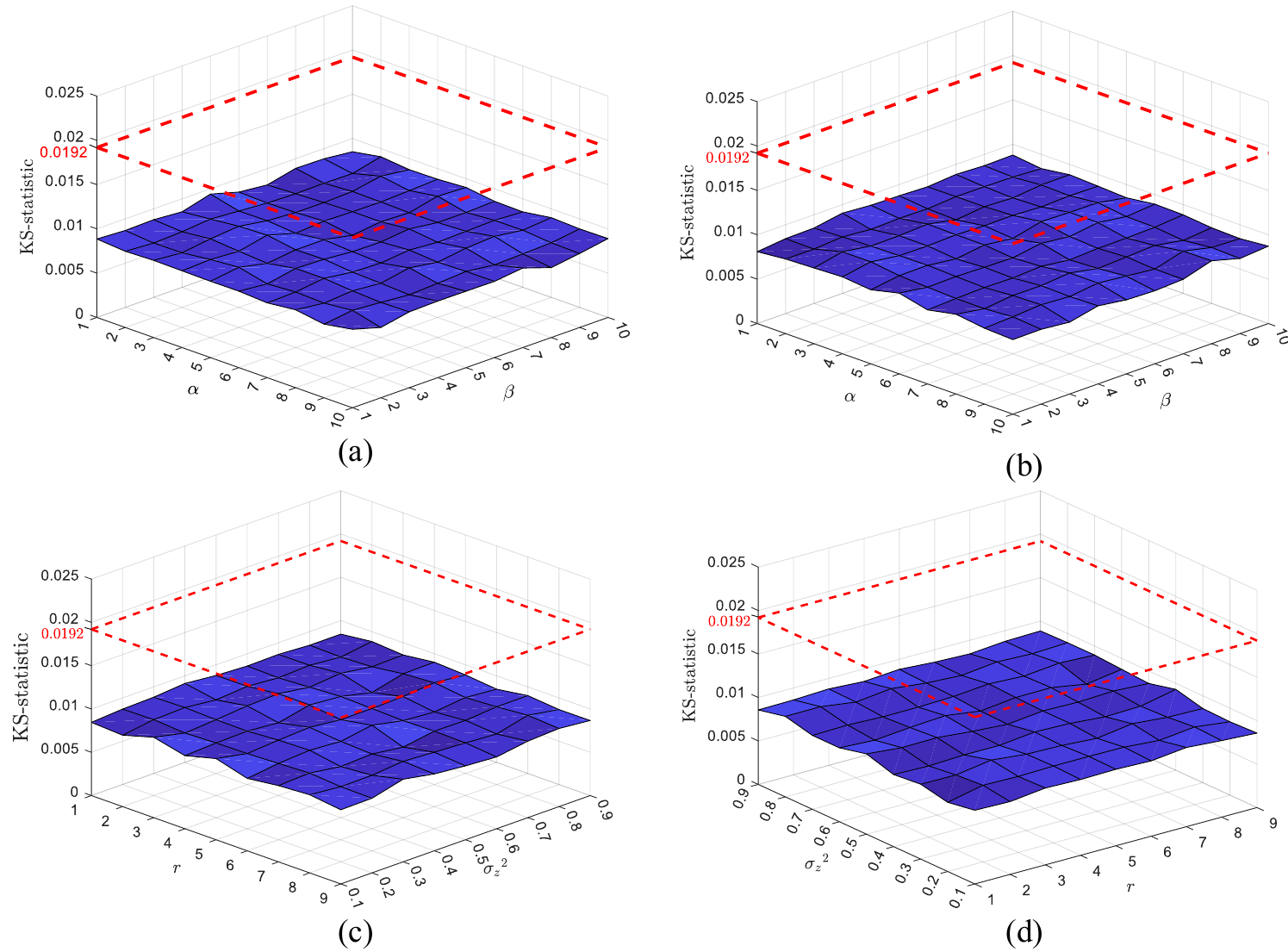}
  \caption{KS-test results for the transformed samples based on QT over different channel parameters, $N = M = 10^4$:
(a) GG channel without additive noise;
(b) GG channel at an SNR of 10dB;
(c) LR channel without additive noise; and
(d) LR channel at an SNR of 10dB.}\label{Figure4}
\end{figure}
\begin{figure}[h]
  \centering
  \includegraphics[width=3.5in]{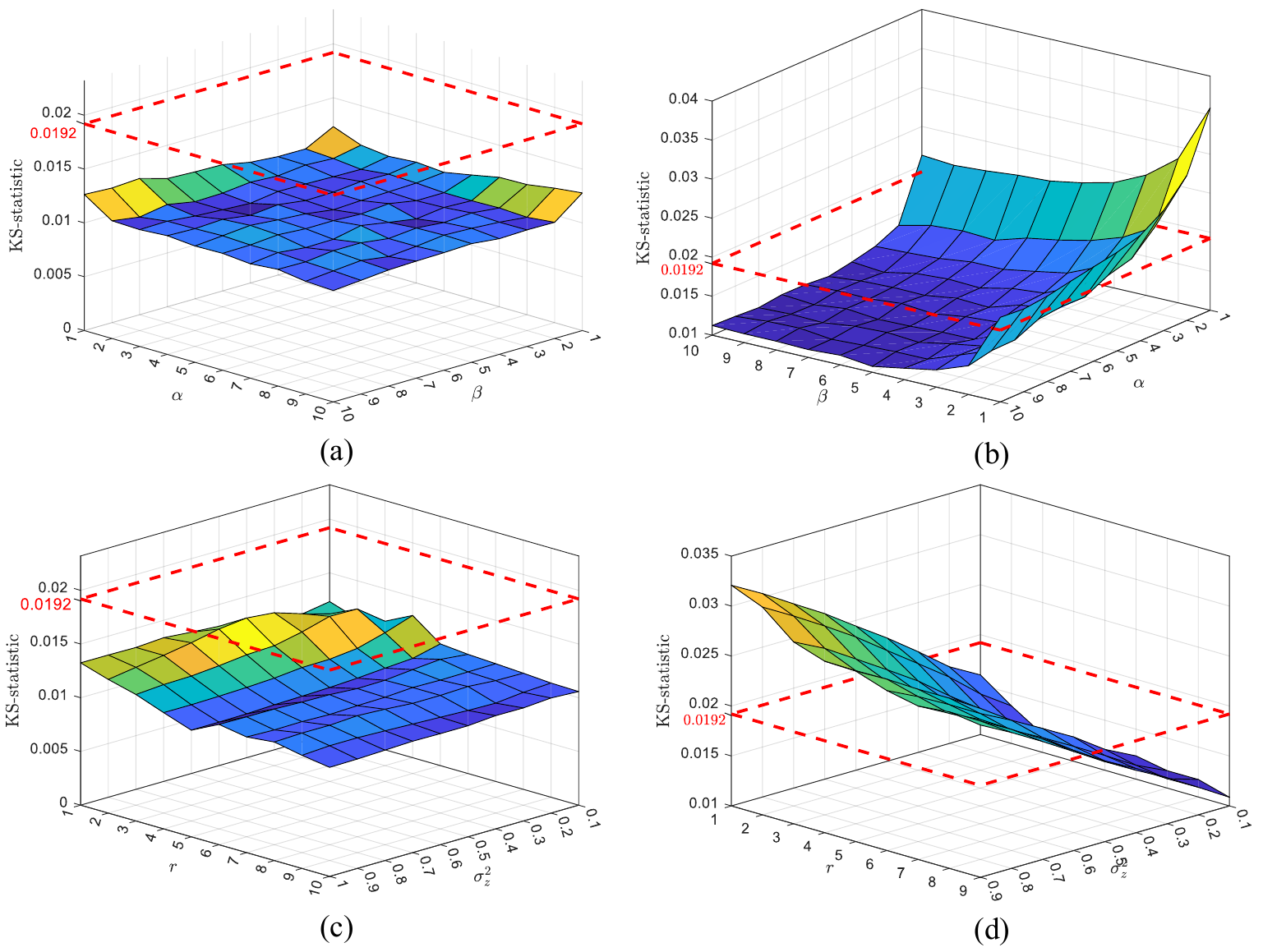}
  \caption{KS-test results for the transformed samples based on Box–Cox over different channel parameters, $N = M = 10^4$:
(a) GG channel without additive noise;
(b) GG channel at an SNR of 10dB;
(c) LR channel without additive noise; and
(d) LR channel at an SNR of 10dB.}\label{Figure5}
\end{figure}

\indent Figs.~\ref{Figure4} and~\ref{Figure5} present the KS-test results for the  transformed samples based on QT and Box–Cox, respectively. The KS statistics are obtained by averaging over 100 independent Monte Carlo realizations. In each realization, the sample sizes of the transformed samples and the Gaussian reference samples are set to $N = M = 10^4$. According to (\ref{eq26}), the corresponding KS critical value is 0.0192, which is also shown in both figures as a benchmark for comparison. \\
\indent Fig.~\ref{Figure4} presents the KS-test results obtained with the QT. Figs.~\ref{Figure4}(a) and~\ref{Figure4}(b) show the results for the GG channel in the
noiseless case and at an SNR of 10~dB, respectively, and Figs.~\ref{Figure4}(c)
and~\ref{Figure4}(d) show those for the LR channel. In all four cases,
the averaged KS statistics remain below the critical value of 0.0192 over
the entire range of channel parameter combinations considered. Moreover,
the KS statistic surfaces vary only slightly with the channel parameters
and are nearly unaffected by the additive noise. These results indicate
that the QT consistently maps the observations of both channel models to
approximately Gaussian samples and is robust to variations in both the
channel parameters and the noise level. \\
\indent Fig.~\ref{Figure5} presents the KS-test results for the transformed samples based on Box–Cox. Although all KS statistics remain below the critical value of 0.0192 under the noiseless condition, they generally increase as the turbulence becomes stronger, indicating a relatively less accurate Gaussian approximation. For the GG channel, this trend is observed as both $\alpha$ and $\beta$ decrease, whereas for the LR channel, it occurs as $r$ decreases and $\sigma_z^2$ increases. At an SNR of 10 dB, the degradation becomes more pronounced, and the KS statistics exceed the critical value for some parameter combinations in the strong-turbulence regions. These results indicate that the Box–Cox transform provides satisfactory Gaussianization under weak-to-moderate turbulence but becomes less robust under strong turbulence, particularly in the presence of additive noise.
\section{The Proposed  Estimators Based on QT and Box-Cox}
Based on the aforementioned discussion in Section II, the physics-informed regularization term consistently enhances the identifiability of the GG parameters by confining the search to a physically feasible region. Therefore, \eqref{eq2} will be incorporated into the  parameter estimation for the GG channels. In contrast, for the LR channels, the effect of the physical constraint  depends on the   turbulence conditions. Accordingly, the LR parameters are estimated without imposing this constraint. \\
\indent Throughout this section, ${Y}_s=\left\{y_{s, i}\right\}_{i=1}^{N_s}$ denote the samples acquired from the considered channel, whereas ${Y}_g=\left\{y_{g, i}\right\}_{i=1}^{N_g}$ denote the samples generated locally under a candidate channel-parameter vector. For the GG channel, $\boldsymbol{\theta}=(\alpha, \beta)$, whereas for the LR channel, $\boldsymbol{\theta} = \left(r, \sigma_z^2\right)$.  Furthermore, $\mathcal{T}_s(\cdot)$ and $\mathcal{T}_g(\cdot)$ denote the transformation models constructed from ${Y}_s$ and ${Y}_g$, respectively. The notation $g \rightarrow s$ indicates that the generated samples are transformed using $\mathcal{T}_s$, whereas $s \rightarrow g$ denotes the reverse operation.
\subsection{QT Estimator}
\indent Given the acquired sample set ${Y}_s$, the proposed
QT estimator determines  $\boldsymbol{\theta}$ by matching ${Y}_s$ with the locally generated
sample set ${Y}_g$ in the
Gaussianized domain. Specifically, the parameter estimation
problem is formulated as
\begin{equation}
\widehat{\boldsymbol{\theta}}_{\mathrm{QT}}
=
\underset{\boldsymbol{\theta}}{\arg\min}\;
\mathcal{J}_{\mathrm{QT}}(\boldsymbol{\theta}),
\label{eq:QT_estimator}
\end{equation}
where
$\mathcal{J}_{\mathrm{QT}}(\boldsymbol{\theta})$ denotes the  objective function to be minimized. To construct
\(\mathcal{J}_{\mathrm{QT}}(\boldsymbol{\theta})\),
two QT models are first constructed separately from
\({Y}_s\) and
\({Y}_g\).
Let \(\widehat{F}_s(\cdot)\) and
\(\widehat{F}_g(\cdot)\) denote their
corresponding empirical CDFs obtained by \eqref{eq9}. Hence, the two QT models are expressed as
\begin{equation}
\mathcal{T}_s^{\mathrm{QT}}(y)
=
\Phi^{-1}\!\left[\widehat{F}_s(y)\right],
\label{eq:QT_model_s}
\end{equation}
and
\begin{equation}
\mathcal{T}_g^{\mathrm{QT}}(y)
=
\Phi^{-1}\!\left[
\widehat{F}_g(y)
\right].
\label{eq:QT_model_g}
\end{equation}
\indent As demonstrated by the KS-test results in Section II, the QT-transformed samples can be well approximated by the standard Gaussian distribution over a wide range of channel parameters. Therefore, even when $Y_s$ and $Y_g$ differ substantially in the original domain, their self-transformed distributions may remain similar, thereby weakening parameter identifiability. To enhance the sensitivity of the transformed distributions to their original mismatch, a cross-transformation strategy is adopted, i.e.,
the locally generated samples are transformed
using the QT model constructed from the acquired samples,
whereas the acquired samples are transformed using the QT
model constructed from the locally generated samples. \\
\indent Furthermore, in each transformation direction, the cumulative probability
of an input sample is evaluated using the QT model constructed
from the opposite sample set. When the input lies between two
adjacent ordered samples, its probability is estimated by simple linear interpolation between the corresponding plotting positions
and then mapped through the inverse standard Gaussian CDF. More specifically, let
\(y_{d,(1)}\leq y_{d,(2)}\leq\cdots\leq y_{d,(N_d)}\)
denote the ordered samples used to construct the QT model,
where \(d\in\{s,g\}\), and let
\(p_k=(k-0.5)/N_d\) denote the corresponding plotting
positions. For
\(y_{d,(k)}\leq y\leq y_{d,(k+1)}\), the interpolated
cumulative probability is given by
\begin{equation}
\widetilde F_d(y)
=
p_k+
\frac{y-y_{d,(k)}}
     {y_{d,(k+1)}-y_{d,(k)}}
\left(p_{k+1}-p_k\right).
\label{eq:QT_linear_interpolation}
\end{equation}

 The
resulting bidirectional cross-QT outputs are given by
\begin{equation}
z_{g\rightarrow s,i}
=
\Phi^{-1}
\left[
\widetilde F_s
\left(y_{g,i}\right)
\right],   i=1,2,\ldots,N_g
\label{eq:QT_g2s}
\end{equation}
and
\begin{equation}
z_{s\rightarrow g,i}
=
\Phi^{-1}
\left[
\widetilde F_g
\left(y_{s,i}\right)
\right], i=1,2,\ldots,N_s
\label{eq:QT_s2g}
\end{equation}
\indent The two cross-QT outputs in \eqref{eq:QT_g2s} and
\eqref{eq:QT_s2g} are then substituted into the standard
Gaussian PDF. The corresponding average negative
log-likelihoods are given by
\begin{equation}
\begin{aligned}
\mathcal{L}_{g\rightarrow s}^{\mathrm{QT}}
(\boldsymbol{\theta})
&=
-\frac{1}{N_g}
\sum_{i=1}^{N_g}
\log
\left[
\frac{1}{\sqrt{2\pi}}
\exp\left(
-\frac{z_{g\rightarrow s,i}^{2}}{2}
\right)
\right] \\
&=
\frac{1}{2}\log(2\pi)
+
\frac{1}{2N_g}
\sum_{i=1}^{N_g}
z_{g\rightarrow s,i}^{2},
\end{aligned}
\label{eq:QT_loss_g2s}
\end{equation}
and
\begin{equation}
\begin{aligned}
\mathcal{L}_{s\rightarrow g}^{\mathrm{QT}}
(\boldsymbol{\theta})
&=
-\frac{1}{N_s}
\sum_{i=1}^{N_s}
\log
\left[
\frac{1}{\sqrt{2\pi}}
\exp\left(
-\frac{z_{s\rightarrow g,i}^{2}}{2}
\right)
\right] \\
&=
\frac{1}{2}\log(2\pi)
+
\frac{1}{2N_s}
\sum_{i=1}^{N_s}
z_{s\rightarrow g,i}^{2}.
\end{aligned}
\label{eq:QT_loss_s2g}
\end{equation}
To ensure that both transformation directions are adequately
matched, the  bidirectional QT loss is defined by retaining
the larger directional negative log-likelihood:
\begin{equation}
\mathcal{L}_{\mathrm{QT}}(\boldsymbol{\theta})
=
\max
\left\{
\mathcal{L}_{g\rightarrow s}^{\mathrm{QT}}
(\boldsymbol{\theta}),
\mathcal{L}_{s\rightarrow g}^{\mathrm{QT}}
(\boldsymbol{\theta})
\right\}.
\label{eq:BQT_basic_loss}
\end{equation}
\indent It can be observed from \eqref{eq:QT_loss_g2s} and
\eqref{eq:QT_loss_s2g} that this loss can be regarded as a second-order matching criterion in the transformed domain. In general, second-order
matching alone is insufficient to fully characterize the
distributional discrepancy. Therefore,
the skewness and kurtosis discrepancies between the two
cross-QT outputs are further incorporated. Moreover, to improve the stability of the objective function, each loss
term is averaged over \(M\) independent realizations. The
repeated realizations also enable reliable estimation of the
variances of the skewness- and kurtosis-matching errors, which
are subsequently used to normalize the two higher-order loss
terms. As a result, the normalized terms have comparable
statistical scales, allowing a single regularization factor $\eta_{\mathrm{reg}}$ to
control their overall contribution.
The final
 loss function is formulated as
\begin{equation}
\begin{aligned}
\mathcal{J}_{\mathrm{QT}}(\boldsymbol{\theta})
=
\frac{1}{M}\sum_{m=1}^{M}
\Bigg\{
&\mathcal{L}_{\mathrm{QT}}^{(m)}
(\boldsymbol{\theta}) +\eta_{\mathrm{reg}}
\left[
\frac{\left(\Delta_{\mathrm{sk}}^{(m)}
(\boldsymbol{\theta})\right)^2}
{\sigma_{\mathrm{sk}}^{2}}\right.\\
&
\left.
+
\frac{\left(\Delta_{\mathrm{ku}}^{(m)}
(\boldsymbol{\theta})\right)^2}
{\sigma_{\mathrm{ku}}^{2}}
\right]
\Bigg\},
\end{aligned}
\label{eq:QT_final_loss}
\end{equation}
where
\begin{equation}
\begin{aligned}
\Delta_{\mathrm{sk}}^{(m)}(\boldsymbol{\theta})
&=
\gamma_{1}
\left(
\{z_{g\rightarrow s,i}^{(m)}\}_{i=1}^{N_g}
\right)
-
\gamma_{1}
\left(
\{z_{s\rightarrow g,i}^{(m)}\}_{i=1}^{N_s}
\right),\\
\Delta_{\mathrm{ku}}^{(m)}(\boldsymbol{\theta})
&=
\gamma_{2}
\left(
\{z_{g\rightarrow s,i}^{(m)}\}_{i=1}^{N_g}
\right)
-
\gamma_{2}
\left(
\{z_{s\rightarrow g,i}^{(m)}\}_{i=1}^{N_s}
\right).
\end{aligned}
\label{eq:QT_moment_difference}
\end{equation}
\begin{equation}
\begin{aligned}
\sigma_{\mathrm{sk}}^{2}(\boldsymbol{\theta})
&=
\frac{1}{M-1}
\sum_{m=1}^{M}
\left[
\Delta_{\mathrm{sk}}^{(m)}(\boldsymbol{\theta})
-
\overline{\Delta}_{\mathrm{sk}}(\boldsymbol{\theta})
\right]^2,\\
\sigma_{\mathrm{ku}}^{2}(\boldsymbol{\theta})
&=
\frac{1}{M-1}
\sum_{m=1}^{M}
\left[
\Delta_{\mathrm{ku}}^{(m)}(\boldsymbol{\theta})
-
\overline{\Delta}_{\mathrm{ku}}(\boldsymbol{\theta})
\right]^2,
\end{aligned}
\label{eq:QT_moment_variances}
\end{equation}
\indent
Then, the objective function in
\eqref{eq:QT_final_loss} is directly employed without the
physics-informed regularization term. For the GG channel, the  objective function is given by
\begin{equation}
\mathcal{J}_{\mathrm{QT}}^{\mathrm{GG}}(\alpha,\beta)
=
\mathcal{J}_{\mathrm{QT}}(\alpha,\beta)
+
\eta_{\mathrm{phy}}
\mathcal{L}_{\mathrm{phy}}(\alpha,\beta),
\label{eq:QT_GG_loss}
\end{equation}
where
\begin{equation}
\mathcal{L}_{\mathrm{phy}}(\alpha,\beta)
=
\left[
\alpha-g\!\left(h^{-1}(\beta)\right)
\right]^2,
\label{eq:GG_physical_loss}
\end{equation}
and \(\eta_{\mathrm{phy}}\) controls the contribution of the
physics-informed regularization term.
\subsection{Box–Cox  Estimator}
\indent Unlike the QT estimator, which is formulated based on
distribution matching, the Box–Cox estimator estimates the
channel parameters by maximizing an approximate log-likelihood
function.  For a  parameter vector
\(\boldsymbol{\theta}\), the locally generated sample set
\({Y}_g(\boldsymbol{\theta})\) is used to determine the
power and shift parameters, denoted by \(\lambda_g\) and
\(\xi_g\), respectively. The resulting Box–Cox transformation
is defined as
\begin{equation}
\mathcal{T}_{g}^{\mathrm{BC}}(y)
=
\begin{cases}
\dfrac{(y+\xi_g)^{\lambda_g}-1}{\lambda_g},
& \lambda_g\neq0,\\[2mm]
\ln(y+\xi_g),
& \lambda_g=0,
\end{cases}
\label{eq:BC_transform}
\end{equation}
where \(y+\xi_g>0\). Let
$
z_{g,i}^{\mathrm{BC}}
=
\mathcal{T}_{g}^{\mathrm{BC}}(y_{g,i}),
 i=1,\ldots,N_g,
$
denote the transformed  generated samples. Their mean
and variance are estimated as
\begin{equation}
\mu_g
=
\frac{1}{N_g}\sum_{i=1}^{N_g}z_{g,i}^{\mathrm{BC}},
\quad
\sigma_g^2
=
\frac{1}{N_g}\sum_{i=1}^{N_g}
\left(z_{g,i}^{\mathrm{BC}}-\mu_g\right)^2.
\label{eq:BC_Gaussian_parameters}
\end{equation}
Note that the transformed generated samples  are  approximated by
\(\mathcal{N}(\mu_g,\sigma_g^2)\) over a wide range of channel conditions, as shown in Fig.~\ref{Figure5}. \\
\indent Since \(\mathcal{T}_{g}^{\mathrm{BC}}(\cdot)\) is monotonic and
invertible, the density in the original domain can be obtained
through the change-of-variables formula. Noting that
\begin{equation}
\left|
\frac{\partial \mathcal{T}_{g}^{\mathrm{BC}}(y)}
{\partial y}
\right|
=
(y+\xi_g)^{\lambda_g-1},
\label{eq:BC_Jacobian}
\end{equation}
the approximate density of an acquired sample under the
candidate model is given by
\begin{equation}
\begin{aligned}
q(y\mid\boldsymbol{\theta})
={}&
\frac{1}{\sqrt{2\pi\sigma_g^2}}
\exp\left[
-\frac{
\left(\mathcal{T}_{g}^{\mathrm{BC}}(y)-\mu_g\right)^2
}{
2\sigma_g^2
}
\right]
(y+\xi_g)^{\lambda_g-1}.
\end{aligned}
\label{eq:BC_approximate_density}
\end{equation}
\indent Then, the average approximate
log-likelihood is
\begin{equation}
\begin{aligned}
\widetilde{\ell}_{\mathrm{BC}}
(\boldsymbol{\theta})
&={}
-\frac{1}{2}\ln(2\pi\sigma_g^2)-
\frac{1}{2N_s\sigma_g^2}
\sum_{i=1}^{N_s}
\left[
\mathcal{T}_{g}^{\mathrm{BC}}(y_{s,i})
-\mu_g
\right]^2\\
&+
\frac{\lambda_g-1}{N_s}
\sum_{i=1}^{N_s}
\ln(y_{s,i}+\xi_g).
\end{aligned}
\label{eq:BC_approximate_log_likelihood}
\end{equation}
It is worth noting that, when the Gaussian approximation in the
transformed domain is sufficiently accurate, the resulting
approximate log-likelihood is expected to closely match the exact
log-likelihood in the original domain, as shown in Fig.~\ref{Figure8}.

\begin{figure}[h]
  \centering
  \includegraphics[width=3in]{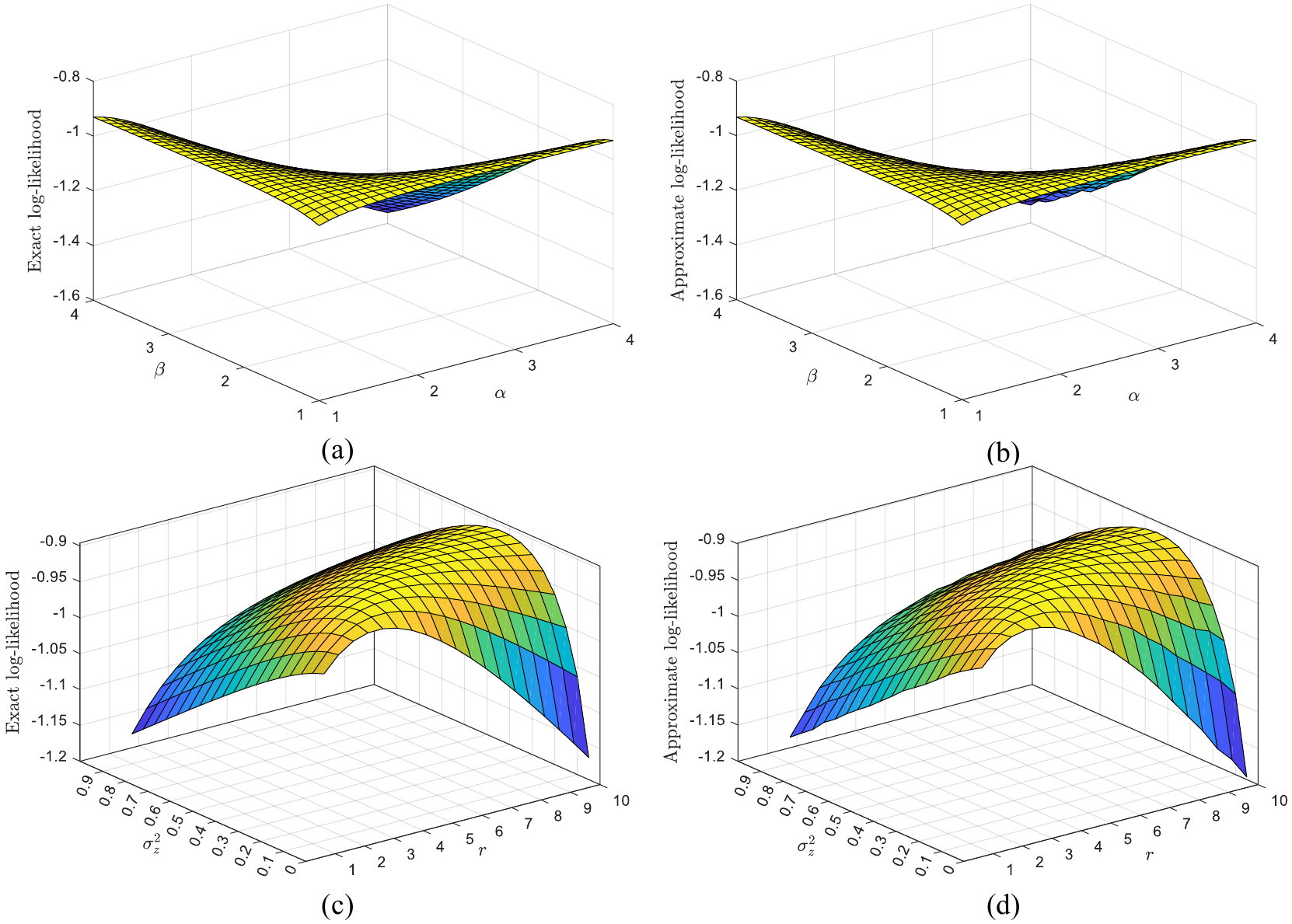}
  \caption{Exact and Box–Cox approximate average
log-likelihood surfaces under noiseless conditions:
(a) exact and (b) approximate results for the GG
channel with \(\alpha=2.34\) and \(\beta=1.02\);
(c) exact and (d) approximate results for the
LR channel with \(r=5\) and
\(\sigma_z^2=0.25\). Here, \(N_s=N_g=10^4\) and \(M=20\).}\label{Figure8}
\end{figure}
\indent Accordingly, the Box–Cox estimator is
obtained as
\begin{equation}
\widehat{\boldsymbol{\theta}}_{\mathrm{BC}}
=
\arg\max_{\boldsymbol{\theta}}
\widetilde{\ell}_{\mathrm{BC}}
(\boldsymbol{\theta}).
\label{eq:BC_estimator}
\end{equation}
Similarly, to reduce the randomness introduced by local sample generation, the approximate log-likelihood is averaged over \(M\)
independent realizations as
\begin{equation}
\mathcal{J}_{\mathrm{BC}}(\boldsymbol{\theta})
=
\frac{1}{M}
\sum_{m=1}^{M}
\widetilde{\ell}_{\mathrm{BC}}^{(m)}
(\boldsymbol{\theta}).
\label{eq:BC_final_objective}
\end{equation}
\indent Finally, the Box–Cox estimator is given by
\begin{equation}
\widehat{\boldsymbol{\theta}}_{\mathrm{BC}}
=
\arg\max_{\boldsymbol{\theta}}
\mathcal{J}_{\mathrm{BC}}(\boldsymbol{\theta}).
\label{eq:BC_final_estimator}
\end{equation}
For the GG channel, the physics-informed objective
function is given by
\begin{equation}
\mathcal{J}_{\mathrm{BC}}^{\mathrm{GG}}(\alpha,\beta)
=
\mathcal{J}_{\mathrm{BC}}(\alpha,\beta)
-
\eta_{\mathrm{phy}}
\mathcal{L}_{\mathrm{phy}}(\alpha,\beta).
\label{eq:BC_GG_objective}
\end{equation}
\section{Simulation Results}
In this section, computer simulations are conducted to evaluate the
 estimation performance of the proposed QT and
Box–Cox estimators for GG and LR
turbulence channels under different turbulence and noise conditions. Throughout the simulations, the locally generated and acquired sample
sets have the same size, i.e., $N_g=N_s=N$. The optimization problems are solved using the genetic algorithm,
with the population size fixed at $100$. \\
\indent For a fair comparison, the
same parameter search ranges are employed for both estimators. For the
GG channel, both $\alpha$ and $\beta$ are searched over
$[0,20]$, subject to $\alpha\geq\beta$. For the LR channel,
$r$ and $\sigma_z^2$ are searched over $[0,10]$ and $[0,1]$,
respectively. To reduce the
variability introduced by local random-sample generation, the
 objective function is averaged over $M=20$ independent
realizations. For the QT estimator, the regularization
coefficient is set to $\eta_{\mathrm{reg}}=10^{-2}$ and for the
GG channel, the coefficient of the physics-informed
regularizer is set to $\eta_{\mathrm{phy}}=1$.
A total of 35 independent trials are conducted to evaluate the
estimation performance in terms of the mean-square error (MSE) and
normalized mean-square error (NMSE). They are defined as
\begin{equation}
\operatorname{MSE}(\boldsymbol{\theta})
=
\mathbb{E}
\left[
\left(
\widehat{\boldsymbol{\theta}}-\boldsymbol{\theta}
\right)^2
\right],
\label{eq:mse_metric}
\end{equation}
and
\begin{equation}
\operatorname{NMSE}(\boldsymbol{\theta})
=
\frac{\operatorname{MSE}(\boldsymbol{\theta})}
{\boldsymbol{\theta}^2},
\label{eq:nmse_metric}
\end{equation}
where $\widehat{\boldsymbol{\theta}}$ denotes the estimate of
$\boldsymbol{\theta}$.

\begin{figure}[h]
  \centering
  \includegraphics[width=3.5in]{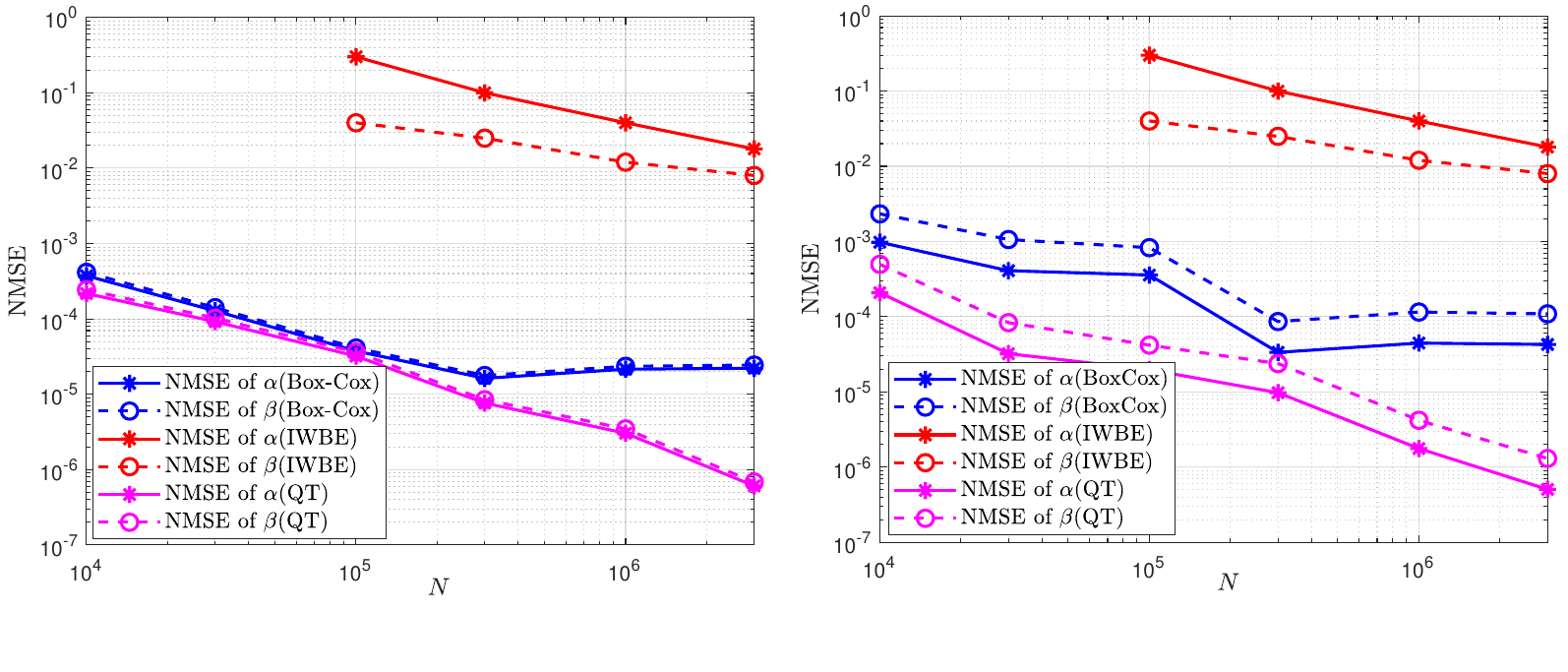}
  \caption{NMSEs of the GG parameter estimates versus the
  sample size at an SNR of $20$ dB,  where the SNR is defined as
$\mathrm{SNR}=\mathbb{E}[I^2]/\sigma_g^2$: (a) weak turbulence with
  $\alpha=21.5890$ and $\beta=19.8208$; (b) moderate turbulence with
  $\alpha=5.4054$ and $\beta=3.7776$.}
  \label{Figure9}
\end{figure}

\indent Fig.~\ref{Figure9} compares the parameter-estimation NMSEs of the
QT, Box–Cox, and IWBE estimators for the GG
channel at an SNR of $20$ dB. Under weak turbulence, the QT and Box–Cox estimators
provide comparable NMSEs at relatively small sample sizes and both
substantially outperform the IWBE. As the sample size increases, the
NMSEs of the QT estimator continue to decrease and reach the
order of $10^{-7}$ at $N=3\times10^6$. In contrast, the Box–Cox
estimator ceases to exhibit further improvement when $N$ exceeds
approximately $3\times10^5$ and approaches an estimation error floor on the order
of $10^{-5}$. This is due to the residual Gaussian approximation error, which dominates the estimation error at large sample sizes.

Under moderate turbulence, the performance difference between the two
proposed estimators becomes more pronounced. The QT estimator
maintains a consistent reduction in the NMSEs of both $\alpha$ and
$\beta$, ultimately achieving NMSEs on the orders of $10^{-7}$ and
$10^{-6}$, respectively, at $N=3\times10^6$. A similar error-floor
behavior can also be observed for the Box–Cox estimator when
the sample size reaches approximately $3\times10^5$, with the NMSEs of
$\alpha$ and $\beta$ leveling off at the orders of $10^{-5}$ and
$10^{-4}$, respectively. The IWBE exhibits considerably higher NMSEs
under both turbulence conditions. These results demonstrate that the
QT estimator is more robust to increased turbulence strength and
benefits more consistently from an increase in the sample size.

\begin{figure}[h]
    \centering
    \includegraphics[width=2.3in]{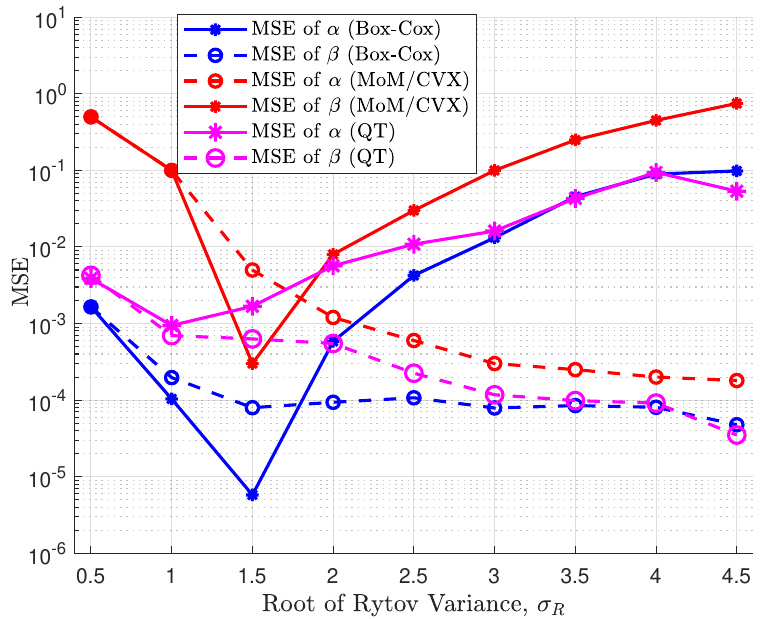}
    \caption{MSE comparison of the Box–Cox, MoM/CVX, and QT estimators versus the root of the Rytov variance $\sigma_R$ under noiseless conditions with $N = 10^5$.}
    \label{fig:noise_free_comparison}
\end{figure}

Figure~\ref{fig:noise_free_comparison} compares the MSEs of the Box–Cox, QT, and MoM/CVX estimators under noiseless conditions as $\sigma_R$ varies. Overall, the Box–Cox and QT estimators outperform the MoM/CVX estimator over most of the considered range. In particular, at $\sigma_R=1.5$, the Box–Cox estimator reduces the MSE of $\alpha$ by nearly two orders of magnitude. For estimating $\alpha$, the Box–Cox estimator performs better under weak-to-moderate turbulence, while the QT estimator becomes comparable at larger $\sigma_R$. For $\beta$, the two estimators provide similar MSEs.

\begin{figure}[h]
\centering
\includegraphics[width=3.5in]{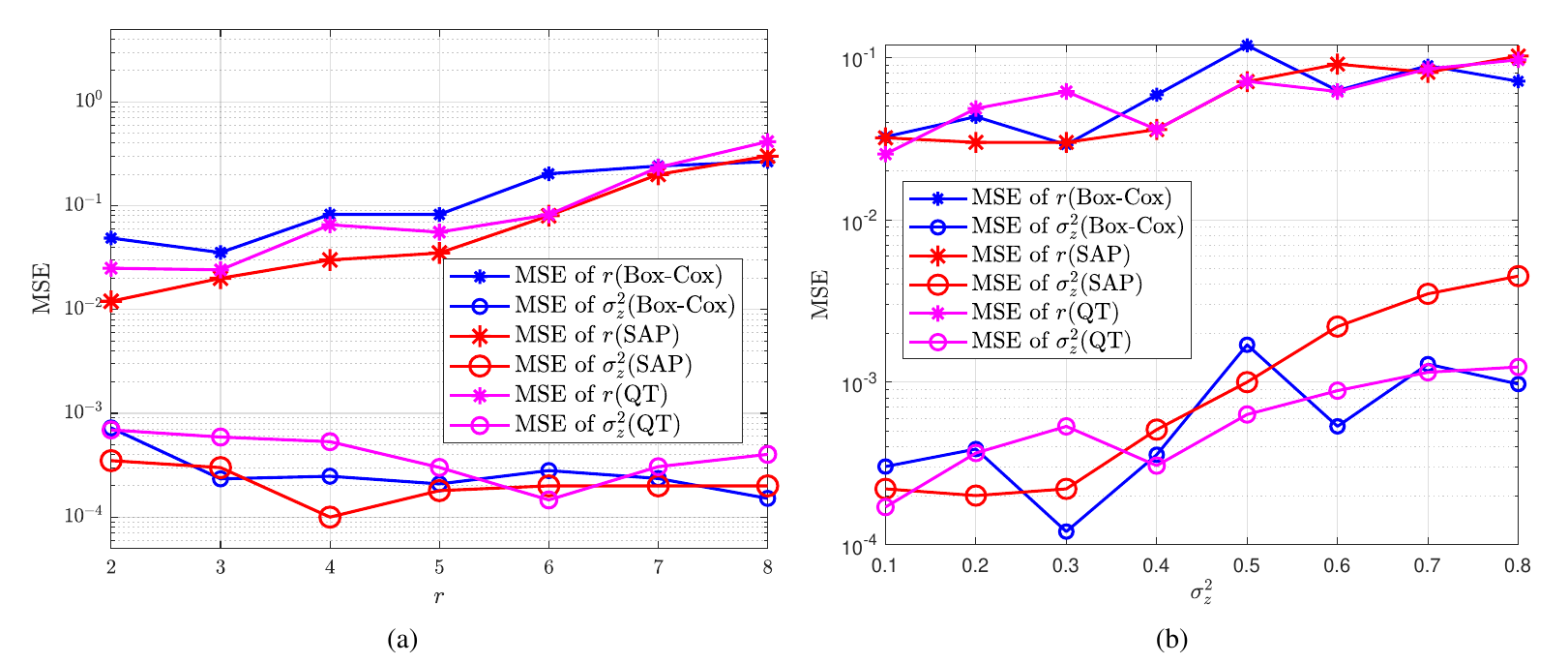}
\caption{MSEs of the LR parameter estimates obtained using the two proposed  estimators and the SAP estimator under  noiseless conditions with $N = 10^4$: (a) MSEs versus the coherence parameter $r$ for fixed $\sigma_z^2=0.25$; (b) MSEs versus the lognormal variance $\sigma_z^2$ for fixed $r=4$.}
\label{fig:LR_noiseless_MSE}
\end{figure}

\indent Fig.~\ref{fig:LR_noiseless_MSE} compares the proposed QT and Box–Cox estimators with the SAP estimator for the LR channel under noiseless conditions. As shown in Fig.~\ref{fig:LR_noiseless_MSE}(a), the three estimators achieve comparable MSE performance over the considered range of $r$. The MSE of the $r$ estimate generally increases with $r$, whereas that of the $\sigma_z^2$ estimate remains on the order of $10^{-4}$, indicating that $r$ becomes more difficult to estimate in the high-coherence regime. A similar conclusion can also be drawn from Fig.~\ref{fig:LR_noiseless_MSE}(b) as $\sigma_z^2$ varies.

\begin{figure}[h]
\centering
\includegraphics[width=2.3in]{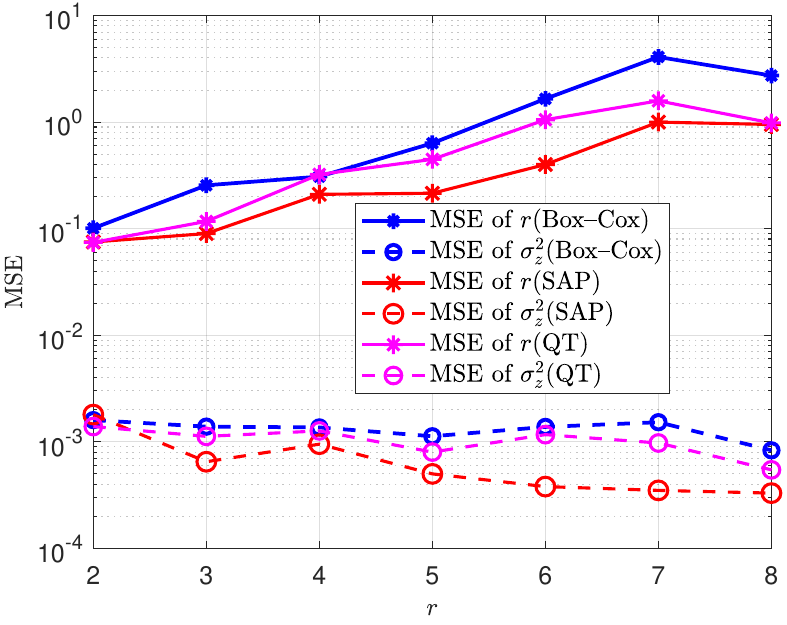}
\caption{MSEs of the LR parameter estimates obtained using the two proposed  estimators and the SAP estimator versus the coherence parameter $r$ at an SNR of $10$~dB, with $\sigma_z^2=0.25, N = 10^4$.}
\label{fig:LR_10dB_MSE}
\end{figure}
\indent To further assess the robustness against additive noise, Fig.~\ref{fig:LR_10dB_MSE} compares the three estimators at an SNR of $10$~dB. As $r$ increases, the three estimators initially exhibit comparable MSEs in estimating $r$. However, when $r>7$, the performance of the Box–Cox estimator deteriorates, resulting in a higher MSE than those of the QT and SAP estimators.
 By contrast, the MSEs of the $\sigma_z^2$ estimates remain on the order of $10^{-3}$ for all three methods. These results indicate that the performance degradation of the Box–Cox estimator under additive noise mainly affects the estimation of $r$ in the high-coherence regime, whereas the QT estimator maintains performance closer to that of the SAP estimator.

\section{Conclusions}
This paper has proposed two Gaussianization methods for estimating the parameters of GG and LR turbulence channels. The first employs QT with bidirectional cross-transformation and higher-order statistical matching, while the second uses the Box–Cox transformation to construct an approximate likelihood incorporating the transformation Jacobian. The KS-test results show that QT consistently produces approximately Gaussian samples under the considered turbulence and noise conditions, whereas Box–Cox achieves satisfactory Gaussianization mainly under weak-to-moderate turbulence. In addition, a physics-informed regularizer derived from the relationship between the channel parameters and the Rytov variance has been introduced to improve the identifiability of the GG parameters. Simulation results demonstrate that both estimators provide robust performance under various turbulence and noise conditions. In particular, the physics-informed regularization term can improve the parameter estimation performance of the GG channel  by at least three orders of magnitude compared with the IWME and MoM/CVX estimator in specific turbulence
scenarios.

\section*{APPENDIX A}

Before proving Proposition~\ref{propskewness1}, we first
establish the following auxiliary lemma, which is given by
\begin{lemma}
\label{lem:gamma_ratio_expansion}
For a fixed real number $q$ and sufficiently large $a$, the asymptotic expansion of $R_q(a)$ can be expressed as
\begin{align}
R_q(a)
& =
\frac{\Gamma(a+q)}{\Gamma(a)a^q}
=
1+\frac{q(q-1)}{2a}
+\frac{q(q-1)(q-2)(3q-1)}{24a^2}
\nonumber\\
&\quad
+\frac{q^2(q-1)^2(q-2)(q-3)}{48a^3}
+\mathcal{O}\!\left(a^{-4}\right).
\label{eq:gamma_ratio_expansion}
\end{align}
\end{lemma}
\begin{IEEEproof}
According to the Stirling series, the $\ln\Gamma(\xi)$ function
has the following asymptotic expansion \cite[Eq.~(5.11.8)]{Olver2010}
\begin{equation}
\ln\Gamma(\xi)
=
\left(\xi-\frac{1}{2}\right)\ln\xi-\xi
+\frac{1}{2}\ln(2\pi)
+\frac{1}{12\xi}
-\frac{1}{360\xi^3}
+\mathcal{O}\!\left(\xi^{-5}\right).
\label{eq:stirling_log_gamma}
\end{equation}
\indent Applying \eqref{eq:stirling_log_gamma} to $\xi=a+q$ and
$\xi=a$, respectively, gives
\begin{align}
\ln R_q(a)
&={}
\left(a+q-\frac{1}{2}\right)\ln(a+q)
-\left(a-\frac{1}{2}\right)\ln a
\nonumber\\
&-q\ln a-q
+\frac{1}{12(a+q)}
-\frac{1}{12a}
+\mathcal{O}\!\left(a^{-4}\right).
\label{eq:log_gamma_ratio_initial}
\end{align}
\indent For sufficiently large $a$, we have
\begin{align}
\ln(a+q)
&=
\ln a+\frac{q}{a}
-\frac{q^2}{2a^2}
+\frac{q^3}{3a^3}
-\frac{q^4}{4a^4}
+\mathcal{O}\!\left(a^{-5}\right),
\label{eq:log_aq_expansion}\\
\frac{1}{a+q}
&=
\frac{1}{a}
-\frac{q}{a^2}
+\frac{q^2}{a^3}
+\mathcal{O}\!\left(a^{-4}\right).
\label{eq:reciprocal_aq_expansion}
\end{align}
\indent Substituting \eqref{eq:log_aq_expansion} and
\eqref{eq:reciprocal_aq_expansion} into
\eqref{eq:log_gamma_ratio_initial}, and performing  some
algebraic manipulations, we have
\begin{equation}
\begin{aligned}
\ln R_q(a)
&=
\frac{q(q-1)}{2a}
-\frac{q(q-1)(2q-1)}{12a^2}
\\
&+
\frac{q^2(q-1)^2}{12a^3}
+\mathcal{O}\!\left(a^{-4}\right).
\label{eq:log_gamma_ratio_final}
\end{aligned}
\end{equation}
For notational convenience, let
\begin{equation}
A=\frac{q(q-1)}{2},
B=-\frac{q(q-1)(2q-1)}{12},
C=\frac{q^2(q-1)^2}{12}.
\end{equation}
\indent Then, exponentiating \eqref{eq:log_gamma_ratio_final} and using the
Taylor expansion of the exponential function gives
\begin{align}
R_q(a)
&=
\exp\left(
\frac{A}{a}
+\frac{B}{a^2}
+\frac{C}{a^3}
+\mathcal{O}\!\left(a^{-4}\right)
\right)
\nonumber\\
&=
1+\frac{A}{a}
+\frac{B+A^2/2}{a^2}
+\frac{C+AB+A^3/6}{a^3}
+\mathcal{O}\!\left(a^{-4}\right).
\label{eq:exponential_gamma_ratio}
\end{align}
Substituting $A$, $B$, and $C$ into
\eqref{eq:exponential_gamma_ratio} yields \eqref{eq:gamma_ratio_expansion}, and this completes the proof.
\end{IEEEproof}
\begin{IEEEproof}[Proof of Proposition~\ref{propskewness1}]
For sufficiently large $\alpha, \beta$,  the
$q$-th  moment of GG random variable $I$ is approximated as
\begin{equation}
M(q) = \mathbb{E}\!\left[I^q\right]
=
\mathbb{E}\!\left[X^q\right]
\mathbb{E}\!\left[Y^q\right]
=
R_q(\tau\beta)R_q(\beta).
\label{eq:GG_moment_Rq}
\end{equation}
Substituting the expansion in
\eqref{eq:gamma_ratio_expansion} into
\eqref{eq:GG_moment_Rq}, and  performing some algebraic manipulations, then the second-, third-
order central moments are respectively obtained as
\begin{equation}
\sigma_{\mathrm{GG}}^2
=
\frac{\tau+1}{\tau\beta}
+\frac{1}{\tau\beta^2},
\label{eq:variance_GG_appendix}
\end{equation}
and
\begin{equation}
\mu_{3,\mathrm{GG}}
=
\frac{2(\tau^2+3\tau+1)}{\tau^2\beta^2}
+\mathcal{O}\!\left(\beta^{-3}\right).
\label{eq:third_central_GG}
\end{equation}
It follows from \eqref{eq:variance_GG_appendix} that
\begin{equation}
\sigma_{\mathrm{GG}}^3
=
\left(\frac{\tau+1}{\tau\beta}\right)^{3/2}
\left[1+\mathcal{O}\!\left(\beta^{-1}\right)\right].
\end{equation}
Therefore, the  skewness of $I$ is
\begin{equation}
\gamma_{1,\mathrm{GG}}
=
\frac{\mu_{3,\mathrm{GG}}}
{\sigma_{\mathrm{GG}}^3}
=
\frac{2(\tau^2+3\tau+1)}
{\sqrt{\tau}(\tau+1)^{3/2}}
\frac{1}{\sqrt{\beta}}
+\mathcal{O}\!\left(\beta^{-3/2}\right).
\label{eq:skewness_GG_appendix}
\end{equation}
Similarly, the fourth-order central moment satisfies
\begin{equation}
\begin{aligned}
\mu_{4,\mathrm{GG}}
={}&
\frac{3(\tau+1)^2}{\tau^2\beta^2}
+
\frac{6(\tau^3+7\tau^2+7\tau+1)}
{\tau^3\beta^3}
\nonumber+
\mathcal{O}\!\left(\beta^{-4}\right).
\label{eq:fourth_central_GG}
\end{aligned}
\end{equation}
Using \eqref{eq:variance_GG_appendix}, the corresponding
fourth-order cumulant becomes
\begin{align}
\kappa_{4,\mathrm{GG}}
&=
\mu_{4,\mathrm{GG}}-3\sigma_{\mathrm{GG}}^4
\nonumber\\
&=
\frac{6(\tau+1)(\tau^2+5\tau+1)}
{\tau^3\beta^3}
+
\mathcal{O}\!\left(\beta^{-4}\right).
\label{eq:fourth_cumulant_GG}
\end{align}
Consequently, the excess kurtosis is
\begin{align}
\gamma_{2,\mathrm{GG}}
&=
\frac{\kappa_{4,\mathrm{GG}}}
{\sigma_{\mathrm{GG}}^4}
\nonumber
=
\frac{6(\tau^2+5\tau+1)}
{\tau(\tau+1)}
\frac{1}{\beta}
+
\mathcal{O}\!\left(\beta^{-2}\right).
\label{eq:excess_kurtosis_GG_appendix}
\end{align}

\end{IEEEproof}
\section*{APPENDIX B}
{
\indent In this appendix, we prove that the LR distribution is positively skewed. According to (\ref{eq4}), the skewness of the LR distribution, $\gamma_{1,LN}$ is given by
\begin{equation}
\begin{aligned}
\gamma_{1,LN} &= \mathbb{E}\left[ \left( \frac{I - \mu}{\sigma_{LN}} \right)^3 \right] = \frac{\mu_{LN,3} - 3\mu_{LN,2} + 2}{(\sigma^2)^{3/2}} \\
&= \frac{1}{\sigma_{LN}^3} \left( \exp(3\sigma_z^2) B(r) - 3\exp(\sigma_z^2) A(r) + 2 \right).
\end{aligned}
\end{equation}
where $\sigma_{LN}$ is the standard deviation of the LR distribution, and
\begin{equation}
B(r) = \frac{r^3 + 9r^2 + 18r + 6}{(1 + r)^3}, \quad A(r) = \frac{r^2 + 4r + 2}{(1 + r)^2}.
\end{equation}
\indent Let $w = \exp(\sigma_z^2) > 0$. The expression for the numerator is defined as
\begin{equation}
f(r, w) = w^3 B(r) - 3w A(r) + 2.
\end{equation}
To prove that $f(r, w) > 0$ for all $r$ and $w$, we first observe that
\begin{equation}
\frac{B(r)}{A(r)} = \frac{r^3 + 9r^2 + 18r + 6}{r^3 + 5r^2 + 6r + 2} > 1.
\end{equation}

Taking the partial derivative of $f(r, w)$ with respect to $w$ and setting it to zero, we have
\begin{equation}
\frac{\partial f(r, w)}{\partial w} = 3w^2 B(r) - 3A(r) = 0.
\end{equation}
The critical point occurs at $w = \sqrt{A(r)/B(r)}$. Since $B(r)/A(r) > 1$, it follows that $0 < \sqrt{A(r)/B(r)} < 1$.

The function $f(r, w)$ is monotonically decreasing for $w \in \left( 0, \sqrt{A(r)/B(r)} \right]$ and monotonically increasing for $w \in \left[ \sqrt{A(r)/B(r)}, \infty \right)$. Given that $w = \exp(\sigma_z^2) \ge 1$, and $B(r)/A(r) > 1$, we have
\begin{equation}
\begin{aligned}
f(r, w) &= w^3 B(r) - 3w A(r) + 2 \\
&> f(r, 1) = B(r) - 3A(r) + 2 = \frac{6r + 2}{(1 + r)^3}.
\end{aligned}
\end{equation}
Since $r \ge 0$, $\frac{6r + 2}{(1 + r)^3} > 0$, it follows that $f(r, w) > 0$.
}
\section*{APPENDIX C}
\begin{IEEEproof}[Proof of Proposition~\ref{propskewness2}]
The $q$-th  moment of LR variable can be
written as
\begin{equation}
M(q)=M_{\mathrm{L}}(q)M_{\mathrm{R}}(q),
\end{equation}
where $M_{\mathrm{L}}(q)$ and $M_{\mathrm{R}}(q)$ denote the
$q$-th  moments of the Lognormal and Rician components,
respectively. Using \eqref{eq4} and expanding the Lognormal moment in powers of $\sigma_z^2$, we have
\begin{align}
M_{\mathrm{L}}(q)
&={}
1+\frac{q(q-1)}{2}\sigma_z^2
+\frac{q^2(q-1)^2}{8}\left(\sigma_z^2\right)^2
\nonumber\\
&+
\frac{q^3(q-1)^3}{48}\left(\sigma_z^2\right)^3
+\mathcal{O}\!\left(\left(\sigma_z^2\right)^4\right).
\label{eq:lognormal_moment_expansion}
\end{align}
Letting $t=1/({1+r})$, the Rician moment can be expressed as
\begin{equation}
M_{\mathrm{R}}(q)
=
(q!)^2
\sum_{j=0}^{q}
\frac{t^j(1-t)^{q-j}}
{j![(q-j)!]^2}.
\end{equation}
Expanding this expression around $t=0$ yields
\begin{align}
M_{\mathrm{R}}(q)
&={}
1+q(q-1)t
+\frac{q(q-1)(q^2-3q+1)}{2}t^2
\nonumber\\
&+
\frac{q(q-1)(q-2)
(q^3-6q^2+8q-1)}{6}t^3
+\mathcal{O}\!\left(t^4\right).
\label{eq:Rician_moment_expansion}
\end{align}

Multiplying \eqref{eq:lognormal_moment_expansion} and
\eqref{eq:Rician_moment_expansion}, we have
\begin{align}
M(q)
&={}
1+\frac{q(q-1)}{2}V_0
+\frac{q^2(q-1)^2}{8}V_0^2
\nonumber\\
&+
\frac{q(q-1)(1-2q)}{2}t^2
+\frac{q^3(q-1)^3}{48}V_0^3
\nonumber\\
&-
\frac{q^2(q-1)^2(2q-1)}{4}V_0t^2
\nonumber\\
&+
\frac{q(q-1)(5q^2-7q+1)}{3}t^3
+\mathcal{O}\!\left(V_0^4\right),
\label{eq:LR_raw_moment_third_order}
\end{align}
where $V_0=\sigma_z^2+2t$. Note that $\sigma_z^2$ and $t$ are both bounded by $V_0$, so all
fourth- and higher-order terms can be uniformly represented by
$\mathcal{O}(V_0^4)$. According to \eqref{eq:LR_raw_moment_third_order},  the
 second-, third-order central moments, and the  fourth-order cumulant
are
\begin{equation}
\begin{aligned}
\mu_{2,\mathrm{LR}}
&=
V_0\left[1+\mathcal{O}\!\left(V_0\right)\right],\\
\mu_{3,\mathrm{LR}}
&=
3\left(V_0^2-2t^2\right)
+\mathcal{O}\!\left(V_0^3\right),\\
\kappa_{4,\mathrm{LR}}
&=
\mu_{4,\mathrm{LR}}-3\mu_{2,\mathrm{LR}}^2\\
&=
8\left(
2V_0^3-9V_0t^2+5t^3
\right)
+\mathcal{O}\!\left(V_0^4\right).
\end{aligned}
\label{eq:LR_central_moment_asymptotic}
\end{equation}
Therefore,
\begin{equation}
\begin{aligned}
\gamma_{1,\mathrm{LR}}
&=
\frac{\mu_{3,\mathrm{LR}}}
{\mu_{2,\mathrm{LR}}^{3/2}}
=
3\sqrt{V_0}\left(1-2\rho^2\right)
+\mathcal{O}\!\left(V_0^{3/2}\right),
\\
\gamma_{2,\mathrm{LR}}
&=
\frac{\kappa_{4,\mathrm{LR}}}
{\mu_{2,\mathrm{LR}}^2}
=
8V_0\left(
2-9\rho^2+5\rho^3
\right)
+\mathcal{O}\!\left(V_0^2\right).
\label{eq:LR_skewness_kurtosis_proof}
\end{aligned}
\end{equation}
where
$
\rho
=
\frac{t}{V_0}
=
\frac{1}{2+(1+r)\sigma_z^2}
$. Since $0<\rho\leq 1/2$, we have
\begin{equation}
1 > 1-2\rho^2\geq\frac{1}{2},
\frac{3}{8}
\leq
2-9\rho^2+5\rho^3
<2.
\end{equation}
Hence, the leading coefficient in
\eqref{eq:LR_skewness_kurtosis_proof} remains bounded.
\end{IEEEproof}
\section*{APPENDIX D}
\begin{IEEEproof}
\emph{1) Skewness under the logarithmic transformation:}
Let
$
X_{\log}=\ln I$
and define its cumulant-generating function as
\begin{equation}
K_{\log}(q)
=
\ln\mathbb{E}\!\left[e^{qX_{\log}}\right]
=
\ln\mathbb{E}\!\left[I^q\right].
\label{eq:log_cumulant_function}
\end{equation}
The skewness of $X_{\log}$ is
\begin{equation}
\gamma_{1,\log}
=
\frac{K_{\log}^{(3)}(0)}
{\left[K_{\log}^{(2)}(0)\right]^{3/2}}.
\label{eq:log_skewness_cumulant}
\end{equation}
where  the superscript $(n)$ denotes the $n$-th derivative. Using \eqref{eqAdd1}, the cumulant-generating function of the
logarithmically transformed GG variable is given by
\begin{equation}
\begin{aligned}
K_{\log,\mathrm{GG}}(q)
&={}
\ln\Gamma(\alpha+q)-\ln\Gamma(\alpha)-q\ln\alpha
\\
&+
\ln\Gamma(\beta+q)-\ln\Gamma(\beta)-q\ln\beta.
\end{aligned}
\end{equation}
Then, according to \cite[Eq.~(5.15.1)]{Olver2010}, we have
\begin{equation}\label{eqAppendixD1}
K_{\log,\mathrm{GG}}^{(2)}(0)
=
\psi_1(\alpha)+\psi_1(\beta)>0,
\end{equation}
and
\begin{equation}\label{eqAppendixD2}
K_{\log,\mathrm{GG}}^{(3)}(0)
=
\psi_2(\alpha)+\psi_2(\beta)<0,
\end{equation}
where $\psi_m(\cdot)$ denotes the polygamma function of order
$m$. Substituting  (\ref{eqAppendixD1}) and (\ref{eqAppendixD2}) into (\ref{eq:log_skewness_cumulant}) gives
\begin{equation}
\gamma_{1,\log,\mathrm{GG}}<0.
\label{eq:GG_log_negative}
\end{equation}

For the LR distribution, the logarithm of the
irradiance is the sum of the logarithms of its independent
Lognormal and Rician components. Since the logarithm of the
Lognormal component is Gaussian, its third-order cumulant is
zero. Thus, the third-order log-cumulant is determined by the
Rician component. Then, according to \eqref{eq:LR_raw_moment_third_order}, we have
\begin{align}
K_{\log,\mathrm{LR}}(q)
&={}
\frac{q(q-1)}{2}V_0
+
\frac{q(q-1)(1-2q)}{2}t^2
\nonumber\\
&+
\frac{q(q-1)(5q^2-7q+1)}{3}t^3
+
\mathcal{O}\!\left(t^4\right),
\end{align}
It follows that
\begin{equation}\label{eqAppendixD3}
K_{\log,\mathrm{LR}}^{(2)}(0)
=
V_0+3t^2+\frac{16}{3}t^3
+\mathcal{O}\!\left(t^4\right)>0,
\end{equation}
and
\begin{align}\label{eqAppendixD4}
K_{\log,\mathrm{LR}}^{(3)}(0)
&=
-6t^2-24t^3+\mathcal{O}\!\left(t^4\right)
\nonumber\\
&=
-6t^2
\left[
1+4t+\mathcal{O}\!\left(t^2\right)
\right]
<0
\end{align}
Hence, substituting \eqref{eqAppendixD3} and \eqref{eqAppendixD4} into (\ref{eq:log_skewness_cumulant}) gives
\begin{equation}
\gamma_{1,\log,\mathrm{LR}}<0.
\label{eq:LR_log_negative}
\end{equation}

\emph{2) Theoretical values of the Box–Cox parameter:}
When $\lambda=1$, the one-parameter Box–Cox transformation
reduces to
\begin{equation}
T_1(I)=I-1,
\end{equation}
which has the same skewness as $I$. Combining \eqref{eq:GG_skewness_asymptotic} and \eqref{eq:LR_skewness_asymptotic} with \eqref{eq:GG_log_negative} and
\eqref{eq:LR_log_negative}, the
intermediate value theorem ensures that, for each
distribution, at least one zero-skewness value exists in
$
0<\lambda<1.
$ We may therefore exclude $\lambda=0$ and consider
$\lambda>0$. In this case,
\begin{equation}
T_\lambda(I)
=
\frac{I^\lambda-1}{\lambda}
\end{equation}
which has the same skewness as $I^\lambda$. Moreover,
\begin{equation}
\mathbb{E}\!\left[(I^\lambda)^k\right]
=
\mathbb{E}\!\left[I^{k\lambda}\right]
=
M(k\lambda).
\label{eq:boxcox_power_moment}
\end{equation}

Using \eqref{eq:boxcox_power_moment} and following derivations
analogous to those in the proofs of
Propositions~\ref{propskewness1} and
\ref{propskewness2}, the transformed skewnesses are obtained
as
\begin{align}
\gamma_{1,\mathrm{GG}}(\lambda)
={}&
\frac{
3\lambda(\tau+1)^2-(\tau^2+1)
}{
\sqrt{\tau}(\tau+1)^{3/2}
}
\frac{1}{\sqrt{\beta}}
+
\mathcal{O}\!\left(\beta^{-3/2}\right),
\label{eq:GG_boxcox_skewness}
\end{align}
and
\begin{equation}
\gamma_{1,\mathrm{LR}}(\lambda)
=
3\sqrt{V_0}
\left(
\lambda-2\rho^2
\right)
+
\mathcal{O}\!\left(V_0^{3/2}\right).
\label{eq:LR_boxcox_skewness}
\end{equation}
Therefore, setting the leading-order terms
\eqref{eq:GG_boxcox_skewness} and
\eqref{eq:LR_boxcox_skewness}  to zero gives
\begin{equation}
\lambda_{\mathrm{GG}}^\star
=
\frac{\tau^2+1}{3(\tau+1)^2}
\end{equation}
and
\begin{equation}
\lambda_{\mathrm{LR}}^\star
=
2\rho^2
=
\frac{2}
{\left[2+(1+r)\sigma_z^2\right]^2}.
\end{equation}
Since $\tau\geq1$ and $0<\rho\leq1/2$, these values satisfy
\begin{equation}
\frac{1}{6}
<
\lambda_{\mathrm{GG}}^\star
<
\frac{1}{3},
\quad
0<
\lambda_{\mathrm{LR}}^\star
<
\frac{1}{2}.
\end{equation}
\end{IEEEproof}

\bibliographystyle{IEEEtran}
\bibliography{sample}

\vfill

\end{document}